\documentclass [a4paper,12pt]{article}
\usepackage{amsmath}
\usepackage{amsfonts}
\usepackage{indentfirst}
\usepackage{esint}
\usepackage{latexsym}
\usepackage{bbm}
\usepackage{amsfonts}
\usepackage{mathrsfs}
\usepackage{amssymb}
\usepackage{amsmath}
\usepackage{amsthm}
\usepackage{latexsym}
\usepackage{indentfirst}
\usepackage{enumitem}
\usepackage{bm}
\usepackage{esint}
\usepackage{hyperref}
\usepackage{cleveref}
\usepackage[numbers,square]{natbib}
\numberwithin{equation}{section}
\usepackage{color}
\allowdisplaybreaks
\usepackage{amsthm}
\usepackage{amssymb}
\usepackage{enumerate}
\usepackage{ulem}
\allowdisplaybreaks

\newtheorem{theorem}{Theorem}[section]
\newtheorem{lemma}[theorem]{Lemma}
\newtheorem{thm}[theorem]{Theorem}

\newtheorem{cor}[theorem]{Corollary}

\newtheorem{rmk}[theorem]{Remark}

\makeatletter
\usepackage{amsmath}
\usepackage{amsfonts}
\usepackage{indentfirst}
\usepackage{esint}
\usepackage{latexsym}
\usepackage{authblk}
\usepackage{color}
\UseRawInputEncoding
\usepackage{amsthm}
\usepackage{amssymb}
\usepackage{enumerate}
\usepackage{ulem}
\makeatletter

\newcommand{\Rmnum}[1]{\expandafter\@slowromancap\romannumeral #1@}
\makeatother
\begin{document}
\title{Quantitative Estimates in Elliptic Homogenization of Non-divergence Form with Unbounded Drift and an Interface}
\author{Yiping Zhang\thanks{Yanqi Lake Beijing Institute of Mathematical Sciences and Applications, Beijing 101408, China and  Yau Mathematical Sciences Center, Tsinghua University, Beijing 100084, China, (zhangyiping161@mails.ucas.ac.cn).}}
\date{}
\maketitle
\begin{abstract}
This paper investigates quantitative estimates in elliptic homogenization of non-divergence form with unbounded drift and an interface, which continues the study of the previous work by Hairer and Manson [Ann. Probab. 39(2011) 648-682], where they investigated the limiting long time/large
scale behaviour of such a process under diffusive rescaling. We determine the effective equation and obtain the size estimates of the gradient of Green functions as well as the optimal (in general) convergence rates. The proof relies on transferring the non-divergence form into the divergence-form with the coefficient matrix decaying exponentially to some (different) periodic matrix on the different sides of the interface first and then investigating this special structure in homogenization of divergence form.

\end{abstract}
\section{Introduction}
In this paper, we consider the non-divergence elliptic equations with unbounded drift, whose junior
coefficient (i.e. drift term) is periodic outside of an ``interface region" of finite thickness, arising from
diffusion process with drift terms under diffusive rescaling. Actually, this paper continues the study of the previous works by Hairer and Manson \cite{MR2789509}, where they investigated the limiting long time/large
scale behaviour of such a process under diffusive rescaling, with the help of the framework provided by
Freidlin and Wentzell \cite{MR1245308} for diffusive process on a graph, in order to identify the generator of the limiting
process. However, we investigate this problem in the sense of PDE instead of the stochastic sense, and we can determine the effective equation and
obtain the optimal $O(\varepsilon)$ convergence rates.

More precisely, for $0<\varepsilon<1$ and $d\geq 3$, we consider the following equation
\begin{equation}\label{1.1}
\tilde{\mathcal{L}}_\varepsilon u_\varepsilon={\tilde{a}}_{ij}^\varepsilon\partial_{ij}u_\varepsilon+\frac{1}{\varepsilon}\tilde{b}_i^\varepsilon\partial_i u_\varepsilon=f\quad \text{ in }\quad \mathbb{R}^d,\\
\end{equation}
where we assume that there exists positive constants $0<\lambda\leq \Lambda$, such that for any $\xi\in \mathbb{R}^d$ and $ \tilde{A}=:(\tilde{a}_{ij})$,
\begin{equation}\begin{aligned}\label{1.2}
\tilde{A}=\tilde{A}^*;\ \lambda|\xi|^2\leq \tilde{A}\xi\cdot\xi\leq \Lambda |\xi|^2;\\
\tilde{A} \text{ is 1-periodic},\  \tilde{A}\in C^\infty(\mathbb{Y}),
\end{aligned}\end{equation}
and the drift term $\tilde{b}$ satisfies
\begin{equation}\label{1.3}
\tilde{b}(y)=\left\{\begin{aligned}&\tilde{b}_+(y)\quad \text{if}\quad y_1>1,\\
&\text{smooth connection},  \text{ if } -1\leq y_1\leq 1,\\
&\tilde{b}_-(y)\quad \text{if}\quad y_1<-1,
\end{aligned}\right.\end{equation}
for $\tilde{b}_+$ and $\tilde{b}_-$ being smooth 1-periodic vector fields.

That $\tilde{a}_{ij}$ is 1-periodic means that $\tilde{a}_{ij}(y+z)=\tilde{a}_{ij}(y)$, for any
$y\in\mathbb{R}^d$ and $z\in \mathbb{Z}^d$. The summation convention is used throughout the paper. Meanwhile, we will denote $\partial_i=:\partial_{x_i}$, $f^\varepsilon(x)=:f(x/\varepsilon)$  and $\mathbb{Y=}:[0,1)^d\cong
\mathbb{R}^d/\mathbb{Z}^d$ if the content is understood. We define $\mathbb{D}=\mathbb{R}\times \mathbb{T}^{d-1}$ with $\mathbb{T}=\mathbb{R}/\mathbb{Z}$, and we say that $u$ is $\mathbb{D}$-periodic if $u$ is 1-periodic in $y'$ for $y=(y_1,y')\in \mathbb{R}^d$. Moreover, we will use the homogenous Sobolev space
$\dot{H}^1(\mathbb{R}^d)=\left\{v:\nabla v\in L^2,||u||_{L^{\frac{2d}{d-2}}}\leq C(d)||\nabla
v||_{L^2}\right\}$ with $d\geq 3$.

Before we move forward, we first introduce some basic results in periodic homogenization of the non-divergence elliptic form with unbounded drift. Precisely, for $0<\varepsilon<1$, define the operator

\begin{equation}\label{1.4}
\tilde{\mathcal{L}}_{\pm,\varepsilon}=:{\tilde{a}}_{ij}^\varepsilon\partial_{ij}\cdot+\frac{1}{\varepsilon}\tilde{b}^\varepsilon_{\pm} \cdot\nabla\cdot.
\end{equation}

It is known in \cite[Chapter 3.3]{bensoussan2011asymptotic} that the invariant measure $m_\pm$ associated with $\tilde{\mathcal{L}}_{\pm,\varepsilon}$ are defined as
\begin{equation}\label{1.5}
\left\{\begin{aligned}
\partial_{y_iy_j}\left(\tilde{a}_{ij}(y) m_\pm(y)\right)-\partial_{y_i}\left(\tilde{b}_{\pm,i}(y)m_\pm(y)\right)=0\text{ in }\mathbb{Y},\\
m_\pm \text{ is 1-periodic, }\int_\mathbb{Y} m_\pm dy=1,\ 0<c_1\leq m_\pm\leq c_2.
\end{aligned}\right.
\end{equation}

It is also known in \cite[Chapter 3.4]{bensoussan2011asymptotic} that under the assumptions \eqref{1.2}-\eqref{1.3}, $f\in L^2(\Omega)$ with $\Omega$ being a bounded $C^{1,1}$ domain in $\mathbb{R}^d$ with $d\geq 2$, and $\int_\mathbb{Y} \tilde{b}_{\pm,i}m_\pm dy=0$, for $i=1,\cdots,d$, then the effective equation for $\tilde{\mathcal{L}}_{\pm,\varepsilon}u_{\pm,\varepsilon} =f$ in $\Omega$, with $u_{\pm,\varepsilon}\in H^1_0(\Omega)$, is given by
\begin{equation}\label{1.6}
\left\{\begin{aligned}
\widehat{\tilde{a}}_{\pm,ij}\partial_{ij}&u_{\pm,0}=f\quad\text{ in }\Omega,&\\
&u_{\pm,0}=0\quad\text{ on }\partial \Omega,&
\end{aligned}\right.\end{equation}
where the effective operator $\widehat{\tilde{a}}_{\pm,ij}$ is defined by
\begin{equation}\label{1.7}
\widehat{\tilde{a}}_{\pm,ij}=\fint_{\mathbb{Y}}\left(\tilde{a}_{ij}+2\tilde{a}_{ik}\partial_{y_k}\tilde{\chi}_{\pm,j}
+\tilde{a}_{k\ell}\partial_{y_k}\tilde{\chi}_{\pm,i}\partial_{y_\ell}\tilde{\chi}_{\pm,j}\right)m_{\pm}(y)dy,\end{equation}
and $\tilde{\chi}_{\pm,j}$, $j=1,\cdots,d$, are the correctors defined by
\begin{equation*}\left\{\begin{aligned}
{\tilde{a}}_{ik}\partial_{y_iy_k}\tilde{\chi}_{\pm,j}+\tilde{b}_{\pm,i}\partial_{y_i} \tilde{\chi}_{\pm,j}=-\tilde{b}_{\pm,j}\text{ in }\mathbb{Y},
\\
\tilde{\chi}_{\pm,j}\text{ is 1-periodic with }\int_{\mathbb{Y}} \tilde{\chi}_{\pm,j}=0.
\end{aligned}\right.
\end{equation*}

See also our previous work \cite{MR}, where by first transferring the non-divergence form \eqref{1.4} into the divergence form, we obtained the $O(\varepsilon)$ convergence
rates and the interior Lipschitz estimates via compactness argument. Moreover, we have provided two examples to show the necessity of the so-called centering conditions and the optimality of the Lipschitz regularity estimates.
See \cite{MR} for more details.

To proceed, for the operator $\tilde{\mathcal{L}}_\varepsilon$ defined in \eqref{1.1} with $d\geq 2$, it is proved in \cite[Proposition 5.5]{MR2789509} that there exists a unique (up to scaling) invariant measure $m$ such that
\begin{equation}\label{1.8}
\left\{\begin{aligned}
\partial_{y_iy_j}\left(\tilde{a}_{ij}(y) m(y)\right)-\partial_{y_i}\left(\tilde{b}_{i}(y)m(y)\right)=0\text{ in }\mathbb{D},\\
m \text{ is $\mathbb{D}$-periodic, } 0<c'_1\leq m\leq c'_2.
\end{aligned}\right.
\end{equation}


We refer to \cite[Chapter 3.3.3]{bensoussan2011asymptotic} for the positivity of the invariant measure $m$. Moreover, with the invariant measure $m$ defined in \eqref{1.8}, we multiply the Equation \eqref{1.1} by $m^\varepsilon(x)$, and set
\begin{equation}\label{1.9}
{a}_{ij}^\varepsilon=\tilde{a}_{ij}^\varepsilon m^\varepsilon,\quad {\beta}_{i}^\varepsilon=\tilde{b}_{i}^\varepsilon m^\varepsilon,\quad \tilde{f}^\varepsilon=f m^\varepsilon.
\end{equation}
We obtain
$${a}_{ij}^\varepsilon\partial_{ij}u_\varepsilon+\frac{1}{\varepsilon}{\beta}_{i}^\varepsilon\partial_i u_\varepsilon=\tilde{f}^\varepsilon\quad \text{ in }\quad \mathbb{R}^d.$$
So that $u_\varepsilon$ is a solution of
\begin{equation}\label{1.10}
\mathcal{L}_\varepsilon u_\varepsilon=:
\partial_i\left({a}_{ij}^\varepsilon\partial_{j}u_\varepsilon\right)
+\frac{1}{\varepsilon}{b}_i^\varepsilon
\partial_i u_\varepsilon =\tilde{f}^\varepsilon\quad\text{ in }\quad\mathbb{R}^d.
\end{equation}
where
\begin{equation}\label{1.11}
{b}_i(y)={\beta}_{i}(y)-\partial_{y_j}{a}_{ij}(y)\quad \text{is }\mathbb{D}\text{-periodic}.
\end{equation}
Then, it follows from \eqref{1.8} and \eqref{1.9} that
\begin{equation}\label{1.12}
\partial_{y_i}{b}_i=0\quad \text{in}\quad\mathbb{D}.
\end{equation}
Moreover, denote the unit cells $C_{\pm,j}$ and $q_{\pm}$ by
\begin{equation*}\begin{aligned}
&C_{+,j}=[j,j+1]\times \mathbb{T}^{d-1},\  C_{-,j}=[-j-1,-j]\times \mathbb{T}^{d-1};\\
&q_{\pm}=\lim_{j\rightarrow\infty}m(C_{\pm,j})=\lim_{j\rightarrow\infty}\int_{C_{\pm,j}}m(y)dy.
\end{aligned}\end{equation*}

Similarly, replacing $m$ and $\tilde{b}$ by $m_\pm$ and $\tilde{b}_\pm$ in \eqref{1.10} and \eqref{1.11}, respectively, gives the definition of the operator $\mathcal{L}_{\pm,\varepsilon}$.

With the notations above at hand, we can determine the effective operator for the operator $\left\{{\mathcal{L}}_\varepsilon\right\}_{1>\varepsilon>0} $, which is stated in the following theorem.
\begin{thm}\label{t1.1}
Under the assumptions \eqref{1.1}-\eqref{1.3} and for any bounded Lipschitz domain $\Omega$ in $\mathbb{R}^d$, $f\in L^2(\Omega)$ and $\int_\mathbb{Y} \tilde{b}_{\pm,i}m_\pm dy=0$, for $i=1,\cdots,d$ with $d\geq 2$, the effective equation for the operator ${\mathcal{L}}_\varepsilon $ defined in \eqref{1.10} is given by
$${\mathcal{L}}_0 u_0=:\operatorname{div}(\widehat{A}(x)\nabla u_0)= f(x)(q_+{I}_{x_1>0}+q_-I_{x_1<0})(x)\text{ in }\Omega,$$
for
\begin{equation}\label{1.13}
\widehat{A}(x)=\left\{\begin{aligned}
q_+\widehat{A_+},\quad\text{ if }\quad x_1>0,\\
q_-\widehat{A_-},\quad\text{ if }\quad x_1<0,
\end{aligned}\right.
\end{equation}
where $\widehat{A_\pm}$ are the homogenized matrices associated with $\mathcal{L}_{\pm,\varepsilon}$, and ${I}_{x_1>0}$ and $I_{x_1<0}$ are the characteristic functions. In general, the matrix $\widehat{A}$ is discontinuous across the interface $\mathcal{I}=\{x:x_1=0\}$.
\end{thm}
\begin{rmk}\label{t1.2}
Actually, the generator of the limiting process of the Equation \eqref{1.1} has also been obtained by the probabilistic method in  \cite{MR2789509}. However, it is not much clear to obtain the effective equation associated with this generator of the limiting process. While in Theorem \ref{t1.1}, we can give an explicit expression for the effective equation.
\end{rmk}
Moreover, we can also obtain the following optimal $O(\varepsilon)$ convergence rates (see \cite{MR4308690} for the optimal $O(\varepsilon)$ convergence rates in general without drift terms in elliptic periodic homogenization of non-divergence form):
\begin{thm}\label{t1.3}
Under the assumptions \eqref{1.1}-\eqref{1.3}, $f\in W^{1,p}(\mathbb{R}^d)\cap L^{2d/(d+4)}(\mathbb{R}^d)$ for  some $p\in (d,\infty)$, and $\int_\mathbb{Y} \tilde{b}_{\pm,i}m_\pm dy=0$, for $i=1,\cdots,d$ with $d\geq 3$, let $u_\varepsilon,u_0\in \dot{H}^1(\mathbb{R}^d)$ solve the equation $\mathcal{L}_\varepsilon u_\varepsilon=f m^\varepsilon$ in $\mathbb{R}^d$ and $\mathcal{L}_0 u_0=f(x)(q_+{I}_{x_1>0}+q_-I_{x_1<0})(x)$  in $\mathbb{R}^d$, respectively, then we have
$$||u_\varepsilon-u_0||_{L^\frac{2d}{d-2}(\mathbb{R}^d)}+||u_\varepsilon-u_0||_{L^\infty(\mathbb{R}^d)}\leq C \varepsilon ||f||_{W^{1,p}(\mathbb{R}^d)\cap L^{2d/(d+4)}(\mathbb{R}^d)},$$
 where the constant $C$ depending only on $d$, $p$, $\tilde{A}$ and $\tilde{b}$.
\end{thm}

Note that in Theorem \ref{t1.3}, we assume $d\geq 3$, since we need the Sobolev embedding inequality to guarantee the existence of $u_\varepsilon$ and $u_0$ in $\dot{H}^1(\mathbb{R}^d)$ and size estimates of the Green functions for the operator $\mathcal{L}_\varepsilon$ and $\mathcal{L}_0$ in $\mathbb{R}^d$, respectively. Moreover, the case $d=2$ will be mentioned in some cases.

Recently, the theory of divergence form in homogenization, such as the $L^2$ and $H^1_0$ convergence rates, Lipschitz estimates, $W^{1,p}$ estimates and the asymptotic behaviours of the Green functions and the fundamental solutions, has been widely understood, especially for the periodic and the ergodic stochastic case, we refer to some excellent expositions \cite{MR3932093,bensoussan2011asymptotic,MR1329546,shen2018periodic} for more details. However, much less is known for the non-divergence form in homogenization, especially for the case with an unbounded drift in homogenization.

Without the unbounded drift, i.e., $\tilde{b}=0$, the first significant qualitative analysis may date back to Avellaneda-Lin \cite{MR978702}, where by using compactness methods, the authors obtained the uniform
$C^{0,\alpha}$, $C^{1,\alpha}$ and $C^{1,1}$ a priori estimates for solutions of boundary value problems of
non-divergence form in elliptic periodic homogenization. Later on, Armsrtong-Lin \cite{MR3665674} have obtained
the quantitative results for the stochastic homogenization for linear uniformly elliptic equations of
non-divergence form. Under strong independence assumptions on the coefficients, they obtained optimal
estimates on the sub-quadratic growth of the correctors with stretched exponential-type bounds in
probability. See also the previous work by Armsrtong-Smart \cite{MR3269637} of the non-divergence form in ergodic stochastic case. Recently, Sprekeler-Tran
\cite{MR4308690} have obtained the optimal $O(\varepsilon)$ convergence rates in general. More precisely, they obtained
$$\text{for any }p\in(1,\infty),\ ||u_\varepsilon-u_0||_{W^{1,p}}=O(\varepsilon),\text{ if }f\in W^{3,q} \text{ for some }q>d.$$

Moreover, see \cite{guo2020conjecture} for the conjecture of an $O(\varepsilon^2)$ convergence rates in
non-divergence form by Guo-Tran in 2-D with some additional structure on the coefficients, and see \cite{guo2022characterizations} for a positive answer of this $O(\varepsilon^2)$ convergence rates by Guo-Sprekeler-Tran. By the way, we
also refer to \cite[Chapter 3.4-3.5]{bensoussan2011asymptotic}, where by using the probabilistic method and
the two-scale expansions, the authors obtained the effective equation and the $O(\varepsilon)$
$L^\infty$-convergence rates under the assumption $f\in W^{3,q}$ for some $q>d$, with the presence of the
unbounded drift.\\

At the end of this section, we briefly explain the idea to obtain the result in Theorem 1.3.

\noindent Step 1 (Section 2). We consider the divergence form $\mathcal{L}_{\varepsilon,>} u_\varepsilon=\partial_i\left(a_{>,ij}(x/\varepsilon)\partial_j u_\varepsilon\right)=f$, where the coefficient matrix $A_>$ is 1-periodic in $y'$ and decays exponentially fast to some periodic matrix $A_+$. For this operator $A_>$, we introduce the definition of correctors, determine the effective operator, and under suitable regularity assumption on $A_>$ and $f$, we can obtain the convergence rates and the uniform interior Lipschitz estimates.\\

\noindent Step 2 (Section 3). We transfer the  divergence form \eqref{1.10} into the divergence form
$\mathcal{L}_\varepsilon u_\varepsilon=\partial_i\left({a}_{1,ij}^\varepsilon\partial_{j}u_\varepsilon\right)
=fm^\varepsilon$ with an interface of finite thickness on the coefficient matrix $A_1$, where $A_1$ decays exponentially fast to $A_+$, if
$y_1\rightarrow +\infty$ and $A_1$  decays exponentially fast to $A_-$, if $y_1\rightarrow -\infty$, for some periodic matrices $A_\pm$. Actually, the
starting point is that, in the periodic case (i.e., there is no ``interface region"), there exists the so-called  flux corrector $\phi$, such that
$\varepsilon^{-1}b^\varepsilon \cdot \nabla u_\varepsilon=\partial_{i}\left(\phi^\varepsilon_{ij}\partial_j
u_\varepsilon\right)$ with $\phi_{ij}=-\phi_{ji}$, if the drift term $b$ is the mean-valued periodic divergence-free vector field, which implies that

\begin{equation}\label{1.14}
\mathcal{L}_\varepsilon u_\varepsilon=\operatorname{div}\left[\left(A^\varepsilon+\phi^\varepsilon\right)\nabla u_\varepsilon\right]=fm^\varepsilon.
\end{equation}

Note that $(A+\phi)\xi\cdot\xi=A\xi\cdot\xi$ for any $\xi\in \mathbb{R}^d$, which would reduce to the case considered in \cite{shen2018periodic} if $m\equiv 1$.\\

\noindent Step 3 (Section 4). For the divergence form obtained in Step 2, we introduce the definition of
correctors, determine the effective operator, and under suitable regularity assumptions, we obtain the size
estimates of the Green function as well as its gradient for this operator, which would imply the desired
convergence rates. Actually, the idea of this proof in step 3 follows from \cite{MR3421758,MR3974127}, since the effective operator
is almost same as the case considered in \cite{MR3974127}.
\section{Basic results for the divergence form}

In this subsection, we consider the second order elliptic equations of divergence form in homogenization. More precisely, for $d\geq 2$ and $0<\varepsilon<1$, consider
\begin{equation}\label{2.1}
\left\{\begin{aligned}
\mathcal{L}_{\varepsilon,>} u_\varepsilon=\partial_i\left(a_{>,ij}(x/\varepsilon)\partial_j u_\varepsilon\right)=f & \quad \text { in } \Omega,\\
u_{\varepsilon}=0&  \quad \text { on } \partial \Omega,
\end{aligned}\right.\
\end{equation}
where the leading coefficient matrix $A_>=(a_{>,ij})$ satisfies the ellipticity condition

\begin{equation}\label{2.2}
\lambda |\xi|^2\leq a_{>,ij}(y)\xi_i\xi_j \leq \Lambda|\xi|^2 \text{\quad for }y\in \mathbb{R}^d \text{ and }\xi\in \mathbb{R}^d,
\end{equation}
where $0<\lambda\leq \Lambda$, and the 1-periodicity condition in $y'=(y_2,\cdots,y_d)$

\begin{equation}\label{2.3}
A_>(y+z)=A_>(y) \text{\quad for }y\in \mathbb{R}^d\text{ and }z=(0,z')\text{ with }z'\in \mathbb{Z}^{d-1},
\end{equation}
and in the direction of $y_1$, there exists constant $\kappa>0$, and  1-periodic matrix-valued function $A_+(y)$ satisfying \eqref{2.2}, such that

\begin{equation}\label{2.4}
|A_>(y)-A_+(y)|\leq \exp\{-\kappa|y_1|\},\ \text{ for }y\in\mathbb{R}^d,
\end{equation}
as well as  the $VMO(\mathbb{R}^d)$ smoothness condition. In order to quantify the smoothness, we will impose the following condition: $A_>\in VMO(\mathbb{R}^d)$ means that

\begin{equation}\label{2.5}
\sup_{x\in \mathbb{R}^d}\fint_{B(x,t)}|A_>-\fint_{B(x,t)}A_>|\leq \rho(t), \text{\quad for }0<t\leq 1,
\end{equation}
where $\rho$ is a nondecreasing continuous function on $[0,1]$ with $\rho(0)=0$. Actually, the conditions \eqref{2.3} and \eqref{2.4} state that $A$ is $\mathbb{D}$-periodic and decays exponentially fast to $A_+$ in $y_1$. Note that a more general case has been investigated in \cite{MR3421758}. Moreover, the correctors $\chi_k$, $k=1,\cdots,d$, are defined as
\begin{equation}\label{2.6}
\left\{\begin{aligned}
\partial_{y_i}\left(a_{>,i j}\left(y\right) {\partial_{y_j}}\chi_{>,k}\right)=- \partial_{y_i}a_{>,ik} \text { in } \mathbb{D}, \\
\chi_{>,k}\  is\  \text{1-periodic in }y'.
\end{aligned}\right.
\end{equation}
The following lemma states the uniqueness and existence of corrector $\chi_{>,k}$.
\begin{lemma}\label{t2.1}
Assume that the matrix $A_>$ satisfies \eqref{2.2}-\eqref{2.4}, then there exists a unique  $\mathbb{D}$-periodic solution $\chi_{>,k}$ to the Equation \eqref{2.6}, such that

\begin{equation}\label{2.7}
\left\{\begin{aligned}
\int_{\mathbb{D}}|\nabla_y(\chi_{>,k}-\chi_{+,k})|^2\leq C(d,\lambda,\Lambda,\kappa);\\
|\chi_{>,k}-\chi_{+,k}|\rightarrow 0 \text{ as }|y_1|\rightarrow +\infty.
\end{aligned}\right.\end{equation}
where $\chi_{+,k}$, defined in \eqref{2.8}, is the corrector for the 1-periodic matrix $A_+$, and the constant $C$ depends only on $d$, $\lambda$, $\Lambda$ and $\kappa$.

\end{lemma}
\begin{proof}For the elliptic periodic matrix $A_+$, the corrector $\chi_{+,k}$, $k=1,\cdots,d$, are defined as
\begin{equation}\label{2.8}
\left\{
\begin{aligned}
\partial_{y_i}\left(a_{+,i j}\left(y\right) {\partial_{y_j}}\chi_{+,k}\right)=- \partial_{y_i}a_{+,ik} \text { in } \mathbb{Y}, \\
\chi_{k}\  is\  \text{1-periodic with} \fint_\mathbb{Y}\chi_{+,k}=0.
\end{aligned}\right.
\end{equation}
Then the difference $\chi_k-\chi_{+,k}$ satisfies the following equation
\begin{equation}\label{2.9}
\begin{aligned}
\partial_{y_i}\left(a_{>,i j}\left(y\right) {\partial_{y_j}}(\chi_{>,k}-\chi_{+,k})\right)&=- \partial_{y_i}(a_{>,ik}-a_{+,ik})-\partial_{y_i}[(a_{>,ij}-a_{+,ij})\partial_{y_i}\chi_{+,k}]\\
&=:\operatorname{div}F(A_>,A_+,\chi_{+,k})\quad \text{ in }\mathbb{D}.
\end{aligned}\end{equation}
For any $1\leq p<\infty$, it is easy to see that
\begin{equation}\label{2.10}
\int_{\mathbb{D}}|a_{>,ik}-a_{+,ik}|^p(y)dy\leq \int_{\mathbb{D}}\exp\{-\kappa p|y_1|\}dy\leq C_p.
\end{equation}
Moreover, using the periodicity of $\chi_+$, we have
\begin{equation}\begin{aligned}\label{2.11}
\int_{\mathbb{D}}|(a_{>,ij}-a_{+,ij})\partial_{y_i}\chi_{+,k}|^p&=\sum_{m=-\infty}^{+\infty}
\int_{[m,m+1]\times \mathbb{T}^{d-1}}|(a_{>,ij}-a_{+,ij})\partial_{y_i}\chi_{+,k}|^pdy\\
&\leq C\sum_{m=-\infty}^{+\infty}\exp\{-\kappa p|m|\}\int_{[m,m+1]\times \mathbb{T}^{d-1}}|\partial_{y_i}\chi_{+,k}|^pdy\\
&\leq C\sum_{m=-\infty}^{+\infty}\exp\{-\kappa p|m|\}||\nabla_y \chi_+||^p_{L^p(\mathbb{Y})}\\
&\leq C_p ||\nabla_y \chi_+||^p_{L^p(\mathbb{Y})}.
\end{aligned}\end{equation}

Then, according to the classical Lax-Milgram Theorem after choosing $p=2$ in \eqref{2.10} and \eqref{2.11}, there exists a solution $\chi_{>,k}-\chi_{+,k}$, solving \eqref{2.9}, or equivalently, there exists a unique solution $\chi_{>,k}$ solving \eqref{2.6}, such that
\begin{equation}\label{2.12}
\left\{\begin{aligned}
\int_{\mathbb{D}}|\nabla_y(\chi_{>,k}-\chi_{+,k})|^2\leq C(d,\lambda,\kappa);\\
|\chi_{>,k}(y)-\chi_{+,k}(y)|\rightarrow 0 \text{ as }|y_1|\rightarrow +\infty.
\end{aligned}\right.\end{equation}

Moreover, if $A_{>}\in C^{0,\alpha}(\mathbb{R}^d)$ and $A_+\in C^{0,\alpha}(\mathbb{R}^d)$ for some fixed $\alpha>0$, then according to the classical Schauder estimates, we have $\nabla_y\chi_>\in C^{0,\alpha}(\mathbb{D})$, due to \eqref{2.9}, \eqref{2.12} and $\nabla_y\chi_+\in C^{0,\alpha}(\mathbb{Y})$.
\end{proof}
In order to introduce the effective operator, if the limit in \eqref{2.13} exists, we first introduce the following notation:
\begin{equation}\label{2.13}
\langle b\rangle=:\lim_{N\rightarrow\infty}\fint_{\mathbb{D}_N}b(y)dy, \text{ for }b \text{ being }\mathbb{D}\text{-periodic},\end{equation}
where $\mathbb{D}_N=:(-N,N)\times \mathbb{T}^{d-1}.$
It is easy to see that
\begin{equation}\label{2.14}
\langle b\rangle:=\fint_{\mathbb{Y}}b(y)dy,\  \text{ if }b \text{ is 1-periodic} \text{ in }y.\end{equation}
Then we can introduce the effective operator $\widehat{A_>}$ for $A_>$:
\begin{equation}\label{2.15}
\widehat{A_>}=:\langle A_>\rangle+\langle A_>\nabla \chi_>\rangle=\langle (a_{>,ij})\rangle+\langle (a_{>,ik}\partial_{y_k}\chi_{>,j})\rangle.
\end{equation}

Actually, the effective operator $\widehat{A_>}$ is determined by the so-called two-scale expansion, which we omit for simplicity, and refer to \cite[Chapter 2.2]{shen2018periodic} for the periodic case.
\begin{lemma}\label{t2.2}
Assume that the matrix $A_>$ satisfies \eqref{2.2}-\eqref{2.4}, then
$\widehat{A_>}=\widehat{A_+}$, where $\widehat{A_+}$ is the effective operator for the periodic matrix $A_+$.
\end{lemma}
\begin{proof}
It is known that $$\widehat{A_+}=\fint_\mathbb{Y} (A_++A_+\nabla \chi_+).$$
Then, according to \eqref{2.4} and \eqref{2.14}, it is easy to see that
\begin{equation}\label{2.16}
\begin{aligned}
\left|\langle A_>\rangle-\langle A_+\rangle\right|
&\leq\lim_{N\rightarrow\infty}\fint_{\mathbb{D}_N}|A_>-A_+|\\
&\leq \lim_{N\rightarrow\infty}\fint_{\mathbb{D}_N}\exp\{-\kappa|y_1|\}\\
&= 0,
\end{aligned}\end{equation}
and
\begin{equation}\label{2.17}\begin{aligned}
\left|\langle A_>\nabla \chi_>\rangle-\langle A_+\nabla\chi_+\rangle\right|&\leq\lim_{N\rightarrow\infty}\fint_{\mathbb{D}_N}\left(|A_>||\nabla\chi_>-\nabla\chi_+|+|A_>-A_+||\nabla\chi_+|\right)\\
&=I_1+I_2,
\end{aligned}\end{equation}
where
\begin{equation}\label{2.18}\begin{aligned}
I_1&\leq \lim_{N\rightarrow\infty}\left(\fint_{\mathbb{D}_N}|A_>|^2\right)^{1/2}\left(\fint_{\mathbb{D}_N}|\nabla\chi_>-\nabla\chi_+|^2\right)^{1/2}\\
&\leq C \lim_{N\rightarrow\infty}{N^{-1/2}}=0,
\end{aligned}\end{equation}
and due to \eqref{2.12},
\begin{equation}\label{2.19}\begin{aligned}
I_2&\leq \lim_{N\rightarrow\infty}\left(\fint_{\mathbb{D}_N}|A_>-A_+|^2\right)^{1/2}
\left(\fint_{\mathbb{D}_N}|\nabla\chi_+|^2\right)^{1/2}\\
&\leq \langle |\nabla\chi_+|^2\rangle ^{1/2}\langle |A_>-A_+|^2\rangle ^{1/2}= 0.
\end{aligned}\end{equation}
Consequently, the desired equality $\widehat{A_>}=\widehat{A_+}$ follows readily from \eqref{2.16}-\eqref{2.19}.
\end{proof}

Moreover, we can obtain the sequence of operators $\mathcal{L}^\ell_{\varepsilon_\ell,>}$ is H-compact in the sense of H-convergence \cite{MR2582099}.
\begin{lemma}\label{t2.3}
Let $\left\{A_{\ell,>}(y)\right\}$ be a sequence of $\mathbb{D}$-periodic matrices satisfying \eqref{2.2}-\eqref{2.4} with the same constants $\lambda$, $\Lambda$ and for some periodic matrix $A_{\ell,+}$ satisfying \eqref{2.2} with the same constant $\kappa$. Let $F_\ell\in H^{-1}(\Omega)$ with $\Omega$ being a bounded Lipschitz domain. Suppose that

\begin{equation}\label{2.20}
\mathcal{L}^\ell_{\varepsilon_\ell,>}(u_\ell)=\operatorname{div}\left(A_{\ell,>}(x/\varepsilon_\ell)\nabla u_\ell\right)=F_\ell\text{ in }\Omega,
\end{equation} where $\varepsilon_\ell\rightarrow 0 $ and $u_{\ell,>}\in H^1(\Omega)$. We further assume that
\begin{equation}\label{2.21}
\left\{\begin{aligned}
&F_\ell\rightarrow F \text{ strongly in }H^{-1}(\Omega),\\
&u_{\ell,>}\rightharpoonup u\text{ weakly in }H^1(\Omega),\\
&\widehat{A_{\ell,>}}\rightarrow A^0,
\end{aligned}\right.\end{equation}
where $\widehat{A_{\ell,>}}$ denotes the effective operator for $\mathcal{L}^\ell_{\varepsilon_\ell,>}$. Then
$A^0$ is a constant matrix satisfying
\begin{equation}\label{2.22}
A^0\xi\cdot\xi\geq \lambda |\xi|^2,\quad\text{for any }\xi\in \mathbb{R}^d,
\end{equation}
and $u\in H^1(\Omega)$ is a weak solution of
\begin{equation}\label{2.23}
\operatorname{div}\left(A^0 \nabla u\right)=F \text{ in }\Omega.
\end{equation}

\end{lemma}
\begin{proof}
The proof is classical and relies on the Div-Curl Lemma \cite[Theorem 2.3.1]{shen2018periodic}. Therefore, we only emphasize on its main ingredient: the matrix $A_>$ admits correctors $\chi_{>,k}$, $k=1,\cdots,d$, such that
\begin{equation*}
\left\{\begin{aligned}
\int_{\mathbb{D}}|\nabla_y(\chi_{>,k}-\chi_{+,k})|^2\leq C(d,\lambda,\kappa);\\
|\chi_{>,k}(y)-\chi_{+,k}(y)|\rightarrow 0 \text{ as }|y_1|\rightarrow +\infty,
\end{aligned}\right.\end{equation*}
and, for any bounded Lipschitz domain $\Omega$, that satisfy the following weak convergence in $L^2(\Omega,\mathbb{R}^d)$,
\begin{equation}\label{2.24}\begin{aligned}
\nabla_y \chi_{>,k}(x/\varepsilon)&= \left(\nabla_y \chi_{>,k}(x/\varepsilon)-\nabla_y \chi_{+,k}(x/\varepsilon)\right)+\nabla_y \chi_{+,k}(x/\varepsilon)\\
&\rightharpoonup 0 \text{ weakly in }L^2(\Omega,\mathbb{R}^d),
\end{aligned}\end{equation}
since $$\begin{aligned}
\int_\Omega|\nabla_y \chi_{>,k}(x/\varepsilon)-\nabla_y \chi_{+,k}(x/\varepsilon)|^2dx&=\varepsilon^d\int_{\varepsilon^{-1}\Omega}|\nabla_y \chi_{>,k}(y)-\nabla_y \chi_{+,k}(y)|^2dy\\
&\leq C \varepsilon \int_{\mathbb{D}}|\nabla_y \chi_{>,k}(y)-\nabla_y \chi_{+,k}(y)|^2dy,
\end{aligned}$$
where we have used $\nabla_y \chi_k(y)-\nabla_y \chi_{+,k}(y)$ is 1-periodic in $y'$. And similarly, we have
\begin{equation}\label{2.25}\begin{aligned}
(A_>\cdot(\nabla_y y_k+\nabla_y \chi_{>,k}))(x/\varepsilon)=&[(A_>-A_+)\cdot(\nabla_y y_k+\nabla_y \chi_{>,k})](x/\varepsilon)\\
&+[A_+\cdot(\nabla_y \chi_{>,k}-\nabla_y \chi_{+,k})](x/\varepsilon)\\
&+[A_+\cdot(\nabla_y y_k+\nabla_y \chi_{+,k})](x/\varepsilon)\\
\rightharpoonup& \widehat{A_+}\nabla_y y_k \text{ weakly in }L^2(\Omega,\mathbb{R}^d).
\end{aligned}\end{equation}
Consequently, the desired property \eqref{2.23} follows from \eqref{2.20}-\eqref{2.21}, \eqref{2.24}-\eqref{2.25} and the Div-Curl Lemma.
\end{proof}

\begin{lemma}\label{t2.4}
Assume that the matrix $A_>$ satisfies \eqref{2.2}-\eqref{2.5}, and we additionally assume that  $A_+$ satisfies the $VMO$ smoothness condition \eqref{2.5}. Then, for any $1<p<\infty$, $\chi_{>,k}-\chi_{+,k}$ satisfies following uniform  $W^{1,p}$ estimates:
\begin{equation}\label{2.26}
||\nabla_y(\chi_{>,k}-\chi_{+,k})||_{L^p(\mathbb{D})}\leq C_p,\end{equation}
where the constant $C_p$ depends only on $d$, $p$, $\lambda$, $\Lambda$, $\kappa$ and $\rho$.
\end{lemma}
\begin{proof}
Since $\chi_{>,k}-\chi_{+,k}$ satisfies the Equation \eqref{2.9} and is $\mathbb{D}$-periodic, then according to the interior $W^{1,p}$ estimates for the VMO coefficients, we have, for any $2< p<\infty$ and $N\in \mathbb{N}$, $\Box_N=:\{x\in\mathbb{R}^d:|x_i|<N,i=1,\cdots,d\}$ and constant $C_p$ depending only on $d$, $p$, $\lambda$, $\Lambda$, $\kappa$ and $\rho$, such that
\begin{equation*}\begin{aligned}
\left(\fint_{\Box_N}|\nabla_y(\chi_{>,k}-\chi_{+,k})|^p\right)^{1/p}
\leq &C_p\left(\fint_{\Box_{2N}}|(a_{>,ij}-a_{+,ij})\partial_{y_i}\chi_{+,k}|^p\right)^{1/p}\\  &\hspace{-4cm}+C_p\left(\fint_{\Box_{2N}}|a_{>,ik}-a_{+,ik}|^p\right)^{1/p}
+C_p\left(\fint_{\Box_{2N}}|\nabla_y(\chi_{>,k}-\chi_{+,k})|^2\right)^{1/2},\\
\end{aligned}\end{equation*}
which, after noting the periodicity in $y'$, is equivalent to
\begin{equation}\label{2.27}\begin{aligned}
\left(\int_{\mathbb{D}_N}|\nabla_y(\chi_{>,k}-\chi_{+,k})|^p\right)^{1/p}\leq & C_pN^{\frac{1}{p}-\frac{1}{2}}\left(\int_{\mathbb{D}_{2N}}|\nabla_y(\chi_{>,k}-\chi_{+,k})|^2\right)^{1/2}\\
&\hspace{-4cm}+C_p \left(\int_{\mathbb{D}_{2N}}|a_{>,ik}-a_{+,ik}|^p\right)^{1/p}
+C_p\left(\int_{\mathbb{D}_{2N}}|(a_{>,ij}-a_{+,ij})\partial_{y_i}\chi_{+,k}|^p\right)^{1/p}.
\end{aligned}\end{equation}
Then, letting $N\rightarrow\infty$ in \eqref{2.27} yields that
\begin{equation}\label{2.28}\begin{aligned}\int_{\mathbb{D}}|\nabla_y(\chi_{>,k}-\chi_{+,k})|^p&\leq C_p \int_{\mathbb{D}}|a_{>,ik}-a_{+,ik}|^p
+\int_{\mathbb{D}}|(a_{>,ij}-a_{+,ij})\partial_{y_i}\chi_{+,k}|^p\\
&\leq C_p,
\end{aligned}\end{equation}
after noting \eqref{2.7}, \eqref{2.10} and \eqref{2.11}. To see $1<p<2$, we use the duality argument. For any $f\in (L^{p'}(\mathbb{D}))^d$ with $\frac{1}{p}+\frac{1}{p'}=1$, we can find a unique $u_R$, up to the addition of a constant, such that $\nabla u_R\in L^2(\mathbb{D})\cap L^{p'}(\mathbb{D})$ solves
\begin{equation*}\partial_i(a^*_{>,ij}\partial_ju_R)=\operatorname{div}f_R\quad \text{ in }\mathbb{D},\end{equation*}
where $f_R=f\eta_R$, and $\eta_R(x_1)=1$, if $|x_1|\leq R$; and $\eta_R(x_1)=0$ if $|x_1|\geq R+1$.
It follows from the $W^{1,p}$ estimates \eqref{2.26} with $2<p'<\infty$, we have
\begin{equation}\label{2.29}||\nabla u_R||_{L^{p'}(\mathbb{D})}\leq C_{p'}||f_R||_{L^{p'}(\mathbb{D})}\leq C_{p'}||f||_{L^{p'}(\mathbb{D})}.\end{equation}
Then,
\begin{equation*}\begin{aligned}
\int_{\mathbb{D}}\nabla_y(\chi_{>,k}-\chi_{+,k})\cdot
f_R&=-\int_{\mathbb{D}}(\chi_{>,k}-\chi_{+,k})\operatorname{div}(f_R)
=-\int_{\mathbb{D}}(\chi_{>,k}-\chi_{+,k})\partial_i(a^*_{>,ij}\partial_ju_R)\\
&=-\int_{\mathbb{D}}\partial_i(a_{>,ij}\partial_j(\chi_{>,k}-\chi_{+,k}))u_R=\int_{\mathbb{D}}F(A_>,A_+,\chi_+)\nabla u_R,
\end{aligned}\end{equation*}with $F(A_>,A_+,\chi_+)$ defined in \eqref{2.9},
which implies that
\begin{equation}\label{2.30}\begin{aligned}
\left|\int_{\mathbb{D}}\nabla_y(\chi_{>,k}-\chi_{+,k})\cdot f_R\right|&\leq C ||\nabla u_R||_{L^{p'}(\mathbb{D})}||F(A_>,A_+,\chi_+)||_{L^{p}(\mathbb{D})}\\
& \leq C ||f||_{L^{p'}(\mathbb{D})}||F(A_>,A_+,\chi_+)||_{L^{p}(\mathbb{D})}.
\end{aligned}\end{equation}

Letting $R\rightarrow \infty$ in \eqref{2.30} and noting $f\in (L^{p'}(\mathbb{D}))^d$ is arbitrary, we have
$$||\nabla_y(\chi_{>,k}-\chi_{+,k})||_{L^p(\mathbb{D})}\leq C_p ||F(A_>,A_+,\chi_+)||_{L^{p}(\mathbb{D})},$$
which yields the desired estimates \eqref{2.26} after noting \eqref{2.10} and \eqref{2.11}.
\end{proof}

\begin{lemma}\label{t2.5}
Under the assumptions in Lemma \ref{t2.4}, we additionally assume that $A_>,A_+\in C^{0,\alpha}(\mathbb{R}^d)$ for some fixed $\alpha>0$, then
\begin{equation}\label{2.31}||\nabla_y(\chi_{>,k}-\chi_{+,k})||_{L^1(\mathbb{D})\cap L^\infty(\mathbb{D})}\leq C,\end{equation} where the constant $C$ depends only on $d$, $\kappa$, $A_>$ and $A_+$.
\end{lemma}
\begin{proof}
According to the classical $C^{1,\alpha}$-estimates, \eqref{2.9} and \eqref{2.12}, we know that $\nabla_y(\chi_{>,k}-\chi_{+,k})\in C^{0,\alpha}_{\text{unif}}(\mathbb{D})$. Then for any $2<p<\infty$,
\begin{equation}\label{2.32}\begin{aligned}
\left(\int_{\mathbb{D}}|\nabla_y(\chi_{>,k}-\chi_{+,k})|^p\right)^{1/p}&\leq ||\nabla_y(\chi_{>,k}-\chi_{+,k})||_{L^\infty (\mathbb{D})}^{1-2/p}\left(\int_{\mathbb{D}}|\nabla_y(\chi_{>,k}-\chi_{+,k})|^2\right)^{1/p}\\
&\leq C,
\end{aligned}\end{equation}
with $C$ being independent of $p>2$. Then, using the duality argument as in Lemma \ref{t2.4}, for any $1>\delta>0$, we have
\begin{equation}\label{2.33}
\int_{\mathbb{D}}|\nabla_y(\chi_{>,k}-\chi_{+,k})|^{1+\delta}\leq C, \text{ with }C \text{ being independent of }\delta,
\end{equation}
which, after letting $\delta\rightarrow 0$ in \eqref{2.33}, implies that $$\int_{\mathbb{D}}|\nabla_y(\chi_{>,k}-\chi_{+,k})|\leq C.$$

Thus, we complete this proof.
\end{proof}
To proceed, we need the solvability  to the following Poisson equation defined in $\mathbb{D}$.
\begin{lemma}\label{t2.6}
Let $f\in L^1(\mathbb{D})\cap L^\infty(\mathbb{D})$, then there exists a solution $u$ to the Poisson equation $\Delta u=f$ in $\mathbb{D}$, such that $\nabla u\in C^{0,\beta'}(\mathbb{D})$, for any $\beta'\in (0,1)$.
\end{lemma}
\begin{proof}Actually, this result has been obtained in \cite{MR3974127}, and we provide it for completeness.
First, we decompose $f(y)=f_1(y_1)+f_2(y)$, with
$$f_1(y_1)=\fint_{\mathbb{T}^{d-1}}f(y_1,y')dy'.$$
It is then easy to see that
\begin{equation}\label{2.34}
f_1,f_2\in L^1(\mathbb{D})\cap L^\infty(\mathbb{D}),\ \fint_{\mathbb{T}^{d-1}}f_2(y_1,y')dy'=0.
\end{equation}
Denote $$N_1(y_1)=\int_{0}^{y_1}\int_{0}^{t}f_1(s)dsdt,$$
then $\Delta N_1(y_1)=\partial_1^2 N_1(y_1)=f_1(y_1)$. Consequently, according to \eqref{2.34} and the regularity theory, $\nabla N_1\in C^{0,\beta'}(\mathbb{D})$, for any $\beta'\in (0,1)$.

To continue, by viewing $y_1$ as a parameter after noting \eqref{2.34}, we can solve the following equation,
\begin{equation}\label{2.35}
\Delta_{y'}N_2(y_1,y')=f_2(y_1,y')\text{ in }\mathbb{T}^{d-1},\ \text{ with }\fint_{\mathbb{T}^{d-1}}N_2(y_1,y')dy'=0.\end{equation}
Define the $\mathbb{D}$-periodic function $g=:(0,\partial_2 N_2,\cdots, \partial_d N_2)$.
 By energy estimates, we have $||\nabla_{y'}N_2(y_1,\cdot)||_{L^2(\mathbb{T}^{d-1})}\leq C ||f_2(y_1,\cdot)||_{L^2(\mathbb{T}^{d-1})}$, with $C$ being independent of $y_1$, which further implies that $||\nabla_{y'}N_2||_{L^2(\mathbb{D})}\leq C ||f_2||_{L^2(\mathbb{D})}$ and $||g||_{L^2(\mathbb{D})}\leq C ||f_2||_{L^2(\mathbb{D})}$.

According to the Equation \eqref{2.35}, we can express $f_2$ as $f_2=\operatorname{div}_yg$ with $g\in L^2({\mathbb{D}})$, then by the
classical Lax-Milgram Theorem, there exists a $\tilde{N}_2$ solving the
Poisson equation $\Delta \tilde{N}_2=\operatorname{div}g$ with $\nabla \tilde{N}_2\in  L^2({\mathbb{D}})$.
Moreover, since $\Delta \tilde{N}_2=f_2$, $\nabla \tilde{N}_2\in  L^2({\mathbb{D}})$ and $f_2\in L^1(\mathbb{D})\cap \in L^\infty(\mathbb{D})$, then according to the $C^{1,\alpha}$ regularity estimates, $\nabla \tilde{N}_2\in C^{0,\beta'}(\mathbb{D})$, for any $\beta'\in (0,1)$.

Consequently, $N_1+\tilde{N}_2$ is the desired solution satisfying $\nabla (N_1+\tilde{N}_2)\in C^{0,\beta'}(\mathbb{D})$, for any $\beta'\in (0,1)$.\end{proof}

With Lemma \ref{t2.5} and Lemma \ref{t2.6} at hand, we introduce the so-called flux corrector $\phi_>$, stated in the following lemma.

\begin{lemma}\label{t2.7}
Under the assumptions in Lemma \ref{t2.5},  denote
\begin{equation}\label{2.36}
B_{>,ij}=\widehat{a}_{>,ij}-a_{>,ij}-a_{>,ik}\partial_{y_k}\chi_{>,j}\quad \text{ in }\mathbb{D}.
\end{equation}

Then there exists the so-called $\mathbb{D}$-periodic flux corrector $\phi_>$, such that
\begin{equation}\label{2.37}\begin{aligned}
B_{>,ij}=\partial_{y_k}\phi_{>,kij}\text{ in }\mathbb{D},\ \phi_{>,kij}=-\phi_{>,ikj},\ \phi_>\in L^\infty(\mathbb{D}).
\end{aligned}\end{equation}
\end{lemma}
\begin{proof}
Similar to the proof of the periodic case, we want to find a matrix function $N_>=(N_{>,ij})$ satisfying
\begin{equation}\label{2.38}
\Delta_y N_{>,ij}=B_{>,ij}\quad  \text{ in }\mathbb{D}.
\end{equation}
First, for the periodic case, it is known that
$$\Delta_y N_{+,ij}=B_{+,ij}=:\widehat{a}_{+,ij}-a_{+,ij}-a_{+,ik}\partial_{y_k}\chi_{+,j}\quad \text{ in }\mathbb{Y}.$$
Then, due to Lemma \ref{t2.2}, $N_{>,ij}-N_{+,ij}$ satisfies
\begin{equation}\label{2.39}\begin{aligned}
\Delta_y (N_{>,ij}-N_{+,ij})&=(a_{+,ij}-a_{>,ij})+(a_{+,ik}-a_{>,ik})\partial_{y_k}\chi_{+,j}
+a_{>,ik}(\partial_{y_k}\chi_{+,j}-\partial_{y_k}\chi_{>,j})\\
&= :\tilde{f}\quad  \text{ in }\mathbb{D}.
\end{aligned}\end{equation}

According to \eqref{2.4} and Lemma \ref{t2.5}, $\tilde{f}\in L^1(\mathbb{D})\cap L^\infty(\mathbb{D})$, then, it follows from Lemma \ref{t2.6} that there exists $N_{>,ij}-N_{+,ij}$ solving the Equation \eqref{2.39} and satisfying $||\nabla_y (N_{>,ij}-N_{+,ij})||_{L^\infty(\mathbb{D})}\leq C$. As a direct consequence, there exists $N_{>,ij}$ solving the Equation \eqref{2.38} and satisfying $||\nabla_y N_{>,ij}||_{L^\infty(\mathbb{D})}\leq C$.

Since $\Delta_y\partial_{y_i}N_{>,ij}=\partial_{y_i}B_{>,ij}=0$ in $\mathbb{D}$ due to \eqref{2.6} and \eqref{2.36}, $N_{>,ij}$ is periodic in $y'$, and $||\nabla_y N_{+,ij}||_{L^\infty(\mathbb{D})}\leq C$,  we know that $\partial_{y_i}N_{>,ij}$ is a bounded harmonic function in $\mathbb{R}^d$ and thus $\partial_{y_i}N_{>,ij}$ is a constant. Let
\begin{equation*}
\phi_{>,kij}=:\partial_{y_k}N_{>,ij}-\partial_{y_i}N_{>,kj}.
\end{equation*}
then
\begin{equation*}
\partial_{y_k}\phi_{>,kij}=\Delta_yN_{>,ij}-\partial_{y_i}\partial_{y_k}N_{>,kj}=B_{>,ij}\quad\text{ in }\mathbb{D},
\end{equation*}
and \begin{equation*}
\phi_{kij}\in L^\infty(\mathbb{D}).
\end{equation*}

Consequently, we complete this proof.
\end{proof}

\begin{lemma}\label{t2.8}
Under the assumptions in Lemma \ref{t2.5},
$||\chi_>||_{L^\infty(\mathbb{D})}\leq C$.
\end{lemma}
\begin{proof} For any $\tilde{y}=(\tilde{y_1},\tilde{y}')\in\mathbb{D}$, and according to the local $L^\infty$-estimates of the Equation \eqref{2.9}, \eqref{2.10}-\eqref{2.12}, we have
$$\begin{aligned}
|\chi_{>,k}(\tilde{y})-\chi_{+,k}(\tilde{y})|\leq& C \int_{(\tilde{y_1}-1,\tilde{y_1}+1)\times\mathbb{T}^{d-1}}\left|\chi_{>,k}(x_1,x')-\chi_{+,k}(x_1,x')\right|dx+C\\
\leq& C \int_{(\tilde{y_1}-1,\tilde{y_1}+1)\times\mathbb{T}^{d-1}}\left|\int^\infty_{x_1}\partial_{y_1}(\chi_{>,k}(s,x')-\chi_{+,k}(s,x'))ds
\right|dx+C\\
\leq& C\int_{(\tilde{y_1}-1,+\infty)\times\mathbb{T}^{d-1}}\left|\partial_{y_1}(\chi_{>,k}-\chi_{+,k})\right|
dx+C,
\end{aligned}$$
which yields the desired estimates after noting $||\chi_{+}||_{L^\infty(\mathbb{D})}\leq C$ and Lemma \ref{t2.5}.
\end{proof}

In view of the homogenization theory in periodic case, we have obtained the all preliminaries to obtain the convergence rates for the matrix $A_>(x/\varepsilon)$.
\begin{thm}[convergence rates]\label{t2.9}
Assume that the matrix $A_>$ satisfies \eqref{2.2}-\eqref{2.4}, and we additionally assume that  $A_>,A_+\in C^{0,\alpha}(\mathbb{R}^d)$ for some $\alpha>0$.
Let $u_{>,\varepsilon}\in H^1_0(\Omega)$, $u_{>,0}\in H^1_0(\Omega)$ be the weak solution to the equation $\mathcal{L}_{>,\varepsilon} u_{>,\varepsilon}=f$ and $\mathcal{L}_{>,0}u_{>,0}=f$ in $\Omega$ with $f\in L^2(\Omega)$ and $\Omega$ being a bounded $C^{1,1}$ domain in $\mathbb{R}^d$, respectively, then we have the following convergence rates estimates:
$$||u_{>,\varepsilon}-u_{>,0}||_{L^2(\Omega)}\leq C \varepsilon ||f||_{L^2(\Omega)},$$
where the constant $C$ depends only on $d$, $\Omega$, $\kappa$, $A_>$ and $A_+$.
\end{thm}

\begin{proof}Denote \begin{equation}\label{2.40}
w_{>,\varepsilon} =u_{>,\varepsilon}-u_{>,0}-\varepsilon\chi_{>,j}^\varepsilon \partial_j u_{>,0},
\end{equation}
according to Lemma \ref{t2.7}, then, a direct computation shows that
\begin{equation}\label{2.41}\begin{aligned}
\mathcal{L}_{>,\varepsilon} w_{>,\varepsilon}&=\mathcal{L}_{>,0} u_{>,0}-\mathcal{L}_{>,0} u_{>,\varepsilon}-\mathcal{L}_{>,\varepsilon}\left(\varepsilon \chi_{>,j}^\varepsilon\partial_j u_{>,0}\right)\\
&=\operatorname{div}\left[(\widehat{A_>}-A^\varepsilon)\nabla u_{>,0}\right]-\operatorname{div}\left(A_>^\varepsilon \nabla_y \chi_{>,j}^\varepsilon \partial_j u_{>,0}\right)-\varepsilon \operatorname{div}\left(A_>^\varepsilon \chi_j^\varepsilon \nabla \partial_j u_{>,0}\right)\\
&=\operatorname{div}\left(B_>^\varepsilon \nabla u_{>,0}\right)-\varepsilon \operatorname{div}\left(A_>^\varepsilon \chi_{>,j}^\varepsilon \nabla \partial_j u_{>,0}\right)\\
&=\varepsilon\partial_k\left( \phi_{>,kij}^\varepsilon \partial_{ij}u_{>,0}\right)-\varepsilon \operatorname{div}\left(A_>^\varepsilon \chi_{>,j}^\varepsilon \nabla \partial_j u_{>,0}\right).
\end{aligned}\end{equation}

It seems that, according to $\chi_>,\phi_>\in L^\infty$, we can obtain
\begin{equation}\label{2.42}
||u_{>,\varepsilon}-u_{>,0}||_{L^2(\Omega)}\leq C \varepsilon ||u_{>,0}||_{H^2(\Omega)}\leq C\varepsilon ||f||_{L^2(\Omega)},\end{equation}
after multiplying the Equation \eqref{2.41} by $w_{>,\varepsilon}$ and integration by parts.

However, it is not rigorous in the above computation, since  $w_{>,\varepsilon}\neq 0$ on the boundary $\partial \Omega$. Actually, if one consider the $\varepsilon$-smoothing method with a boundary cut-off function, one can
obtain the $O(\varepsilon^{1/2})$ convergence rates, and recover the $O(\varepsilon)$ convergence rates by
duality argument, which we omit for simplicity. For readers' convenience, we refer to \cite[Thm. 3.3.2, Thm. 3.4.3]{shen2018periodic} for more details.
\end{proof}

Moreover, we can also obtain the following interior Lipschitz estimates.
 \begin{thm}[interior Lipschitz estimates]\label{t2.10}
Under the assumptions in Theorem \ref{t2.9}, let $u_{>,\varepsilon}\in H^1(B)$ be a weak solution to the equation
${\mathcal{L}}_{>,\varepsilon} u_{>,\varepsilon}=f$ with $f\in L^p(B)$ for some $p>d$ and some ball  $B=B(x_0,2R)$ with $x_0\in \mathbb{R}^d$ and $R>0$, then
$$||\nabla u_{>,\varepsilon}||_{L^\infty (B(x_0,R))}\leq C_p\left\{\left(\fint_B |\nabla u_{>,\varepsilon}|^2\right)^{1/2}+R\left(\fint_B |f|^p\right)^{1/p}\right\},$$
where $C$ depends only on $d$, $p$, $\kappa$, $A_>$ and $A_+$.
\end{thm}
\begin{proof}
The proof of Theorem \ref{t2.10} is based on the method of compactness argument and it is done in the following three steps:\\

\noindent Step 1. [One-step improvement]. We take advantage of the uniform H-convergence of the multi-scale
problem $\mathcal{L}_{>,\varepsilon} u_{>,\varepsilon}=f$ in $B$ to the homogeneous effective problem
$\mathcal{L}_{>,0} u_{>,0}=f$ in $B$, which states that the multi-scale solution $u_{>,\varepsilon}$
inherits the medium-scale regularity of the solution $u_{>,0}$. In this step, we use Lemma \ref{t2.3}, interior Caccioppoli's inequality for $u_{>,\varepsilon}$ and the $C^{1,\alpha}$ regularity estimates for $u_{>,0}$.\\

\noindent Step 2. [Iteration]. The previous estimates can be iterated to obtain Lipschitz regularity of $u_{>,\varepsilon}$ down to scale $\varepsilon$. In this step, we need to notice the scaling property of $\mathcal{L}_{>,\varepsilon}$ and interior Caccioppoli's inequality for $u_{>,\varepsilon}$.\\

\noindent Step 3. [A blow-up argument]. We use the regularity result of $\mathcal{L}_{>,1}$ to obtain the Lipschitz regularity on scales smaller than $\varepsilon$ due to $A_>\in C^{0,\alpha}$.\\

Since all of the operations above are totally similar to the proof of \cite[Theorem 4.1.1]{shen2018periodic}, we omit it for simplicity, and refer to \cite[Theorem 4.1.1]{shen2018periodic} for the details. \end{proof}

\section{Transferring non-divergence form into divergence form}
In this section, we investigate the non-divergence elliptic equation \eqref{1.1} with unbounded drift in
periodic homogenization with an interface, arising from diffusion process with drift terms under diffusive
rescaling.
Recall that in Section 1, we have introduced the following notations:
\begin{equation}\label{3.1}\begin{aligned}
C_{+,j}=[j,j+1]\times \mathbb{T}^{d-1},\  C_{-,j}=[-j-1,-j]\times \mathbb{T}^{d-1};\
q_{\pm}=\lim_{j\rightarrow\infty}m(C_{\pm,j}).
\end{aligned}\end{equation}

As pointed out in \cite{MR2789509}, it is
straightforward to adapt the proofs in \cite{MR2789509} to cover the case of nonconstant diffusivity as well. Then it follows from \cite[Proposition 5.5]{MR2789509} that
$$\begin{aligned}&\left|\int_{(k,k+1)\times \mathbb{T}^{d-1}}(m(y)-q_+m_+(y))dy\right|+\left|\int_{(-k,-k+1)\times \mathbb{T}^{d-1}}(m(y)-q_-m_-(y))dy\right|\\
&\hspace{2cm}\leq C \exp\{-Ck\}\quad \text{ as }\quad k\rightarrow+\infty.\end{aligned}$$

Actually, it follows from  \cite[Proposition 5.5]{MR2789509} and the definition of the total variation of $m-q_\pm m_\pm$, we have
\begin{equation}\label{3.2}\begin{aligned}&\int_{(k,k+1)\times \mathbb{T}^{d-1}}\left|m(y)-q_+m_+(y)\right|dy+\int_{(-k,-k+1)\times \mathbb{T}^{d-1}}\left|m(y)-q_-m_-(y)\right|dy\\
&\hspace{2cm}\leq C \exp\{-Ck\}\quad \text{ as }\quad k\rightarrow+\infty.\end{aligned}\end{equation}

Moreover, according to \eqref{1.3}, \eqref{1.5} and \eqref{1.8}, it is easy to see that $m(y)-q_+m_+(y)$ satisfies
\begin{equation}\label{3.3}
\partial_{y_iy_j}\left[\tilde{a}_{ij}(y) (m-q_+m_+)\right]-\partial_{y_i}\left[\tilde{b}_{i}(y)\left(m-q_+m_+\right)\right]=0\quad \text{ if }y_1>1,
\end{equation}
then it follows from the $L^\infty$-estimates,  Lipschitz regularity estimates and \eqref{3.2} that
\begin{equation}\label{3.4}
\left|(m-q_+m_+)(y)\right|+\left|\nabla_y(m-q_+m_+)(y)\right|\leq C \exp\{-C|y_1|\},\ \text{ for }y_1\rightarrow+\infty,
\end{equation}
and similarly, we have
\begin{equation}\label{3.5}
\left|(m-q_-m_-)(y)\right|+\left|\nabla_y(m-q_-m_-)(y)\right|\leq C \exp\{-C|y_1|\},\ \text{ for }y_1\rightarrow-\infty.
\end{equation}

Next, multiplying the Equation \eqref{1.10} by $u_\varepsilon$ and integrating the resulting equation over $\mathbb{R}^d$ after using \eqref{1.12}, we have
\begin{equation}\label{3.6}\begin{aligned}
||\nabla u_\varepsilon||^2_{L^2(\mathbb{R}^d)}&\leq C ||f||_{L^{\frac{2d}{d+2}}(\mathbb{R}^d)}||u_\varepsilon||_{L^{\frac{2d}{d-2}}(\mathbb{R}^d)}\\
&\leq C ||f||_{L^{\frac{2d}{d+2}}(\mathbb{R}^d)}||\nabla u_\varepsilon||_{L^{2}(\mathbb{R}^d)},
\end{aligned}\end{equation}
where the constant $C$ depends only on $d$, $\lambda$ and $\Lambda$. Therefor, by the classical Lax-Milgram theorem, there exists a solution $u_\varepsilon$ to the Equation \eqref{1.10}, if $f\in L^{\frac{2d}{d+2}}(\mathbb{R}^d)$ with $d\geq 3$, such that $||u_\varepsilon||_{\dot{H}^{1}(\mathbb{R}^d)}\leq C ||f||_{L^{\frac{2d}{d+2}}(\mathbb{R}^d)}$. Recall that we have used the homogenous Sobolev space $\dot{H}^1(\mathbb{R}^d)=\left\{v:\nabla v\in L^2,||u||_{L^{\frac{2d}{d-2}}}\leq C(d)||\nabla
v||_{L^2}\right\}$ with $d\geq 3$.

In the following content, our main effort is to transfer the Equation \eqref{1.10} into the divergence form \eqref{2.1} considered in Section 2. Then in view of \eqref{2.1} and \eqref{2.4}, we need that $\phi$, associated with \eqref{1.14}, decays exponentially fast in $y_1$ to some 1-periodic matrix $\phi_+$ if $y_1\rightarrow \infty$; and it is similar for the case if $y_1\rightarrow-\infty$. Then, we introduce the following results.

\begin{lemma}\label{t3.1}
Assume $f$ decays  exponentially fast in $y_1$, i.e. $|f(y)|\leq C\exp\{-C|y_1|\}|$, $f$ is $\mathbb{D}$-periodic and $\text{supp}(f)\subset [0,\infty]\times \mathbb{T}^{d-1}$ with $\int_{(0,\infty)\times \mathbb{T}^{d-1}}fdy=0$. Then there exists a $\mathbb{D}$-periodic solution $u$ to the Poisson equation
$\Delta u=f$ in $\mathbb{D}$, such that $\nabla u\in C^{0,\beta'}(\mathbb{D})$, for any $\beta'\in (0,1)$, and $\nabla u$ decays  exponentially fast in $y_1$.
\end{lemma}

\begin{proof}
Similar to the proof of Lemma \ref{t2.6}, we decompose $f(y)=f_1(y_1)+f_2(y)$, with
$$f_1(y_1)=\fint_{\mathbb{T}^{d-1}}f(y_1,y')dy'.$$

\noindent
It is then easy to see that
\begin{equation}\label{3.7}\begin{aligned}
f_1,f_2\text{ decays  exponentially fast in }y_1,\  \text{supp}(f_2)\subset [0,\infty]\times \mathbb{T}^{d-1},\\
\text{supp}(f_1)\subset [0,\infty],\ \int_{0}^{+\infty}f_1(y_1)dy_1=0,\ \fint_{\mathbb{T}^{d-1}}f_2(y_1,y')dy'=0.
\end{aligned}\end{equation}

\noindent
Denote $$N_1(y_1)=\int^{+\infty}_{y_1}\int^{+\infty}_{t}f_1(s)dsdt,$$
then $\Delta N_1(y_1)=\partial_1^2 N_1(y_1)=f_1(y_1)$. Consequently, according to \eqref{3.7} and the regularity theory, $\nabla N_1\in C^{0,\beta'}(\mathbb{D})$, for any $\beta'\in (0,1)$. Moreover,
\begin{equation}\label{3.8}\begin{aligned}\partial_1 N_1(y_1)=-\int^{+\infty}_{y_1}f_1(s)ds=-\int^{+\infty}_{0}f_1(s)ds=0,\ \text{ if }y_1\leq 0;\\
\left|\partial_1 N_1(y_1)\right|\leq \int^{+\infty}_{y_1}|f_1(s)|ds\leq C\exp\{-C|y_1|\},\ \text{ if }y_1> 1.
\end{aligned}\end{equation}

To proceed, by viewing $y_1$ as a parameter after noting \eqref{3.7}, we can solve the following equation,
\begin{equation}\label{3.9}\Delta_{y'}N_2(y_1,y')=f_2(y_1,y')\text{ in }\mathbb{T}^{d-1},\ \text{ with }\fint_{\mathbb{T}^{d-1}}N_2(y_1,y')dy'=0.\end{equation}

Similar to the explanation of Lemma \ref{t2.6}, there exists a solution $\tilde{N}_2$  to the Poisson equation $\Delta \tilde{N}_2=f_2$, satisfying $\nabla \tilde{N}_2\in  L^2({\mathbb{D}})$ and $\nabla \tilde{N}_2\in C^{0,\beta'}(\mathbb{D})$, for any $\beta'\in (0,1)$. Moreover, multiplying the equation $\Delta \tilde{N}_2=f_2$ by $\tilde{N}_2$ and integrating over $(R_1,R_2)\times \mathbb{T}^{d-1}$ for any $R_2>R_1\geq 1$, yields that
\begin{equation}\label{3.10}\begin{aligned}
\int_{(R_1,R_2)\times \mathbb{T}^{d-1}}\nabla \tilde{N}_2\cdot \nabla \tilde{N}_2=&\int_{\{R_2\}\times \mathbb{T}^{d-1}}\partial_{y_1} \tilde{N}_2\cdot \tilde{N}_2-\int_{\{R_1\}\times \mathbb{T}^{d-1}}\partial_{y_1} \tilde{N}_2\cdot \tilde{N}_2\\
&-\int_{(R_1,R_2)\times \mathbb{T}^{d-1}} f_2\tilde{N}_2\\
=&:I_1-I_2-I_3.
\end{aligned}\end{equation}

Next, multiplying the equation $\Delta \tilde{N}_2=f_2$ by $1$ and integrating over $(R_1,R_2)\times \mathbb{T}^{d-1}$ for any $R_2>R_1\geq 1$, yields that
\begin{equation*}
\int_{\{R_2\}\times \mathbb{T}^{d-1}}\partial_{y_1} \tilde{N}_2-\int_{\{R_1\}\times \mathbb{T}^{d-1}}\partial_{y_1} \tilde{N}_2=-\int_{(R_1,R_2)\times \mathbb{T}^{d-1}}f_2=0,
\end{equation*}
which implies that
\begin{equation}\label{3.11}
\int_{\{R\}\times \mathbb{T}^{d-1}}\partial_{y_1} \tilde{N}_2=0,
\end{equation} for any $R\geq 1$,
due to $\nabla \tilde{N}_2\in L^2(\mathbb{D})$. Denote $C(R,\tilde{N}_2)=:\fint_{\mathbb{T}^{d-1}} \tilde{N}_2(R,y')dy'$, then
$$I_1=\int_{\{R_2\}\times \mathbb{T}^{d-1}}\partial_{y_1} \tilde{N}_2\cdot \tilde{N}_2=\int_{\{R_2\}\times \mathbb{T}^{d-1}}\partial_{y_1} \tilde{N}_2\cdot (\tilde{N}_2-C(R_2,\tilde{N}_2)),$$
which, after using Poincar\'{e} inequality, implies that
\begin{equation}\label{3.12}
|I_1|\leq C \int_{\{R_2\}\times \mathbb{T}^{d-1}}|\nabla_y \tilde{N}_2|^2.
\end{equation}

Similarly,
\begin{equation}\label{3.13}
|I_2|\leq C \int_{\{R_1\}\times \mathbb{T}^{d-1}}|\nabla_y \tilde{N}_2|^2.
\end{equation}

To estimate $I_3$, after noting \eqref{3.7}, we have,
\begin{equation*}\begin{aligned}
I_3&=\int_{R_1}^{R_2}\int_{\mathbb{T}^{d-1}} f_2(y_1,y')\tilde{N}_2(y_1,y')dy'dy_1\\
&=\int_{R_1}^{R_2}\int_{\mathbb{T}^{d-1}} f_2(y_1,y')\left(\tilde{N}_2(y_1,y')-C(y_1,\tilde{N}_2)\right)dy'dy_1,
\end{aligned}\end{equation*}
which, after using Poincar\'{e} inequality and Holder inequality, implies that
\begin{equation}\label{3.14}\begin{aligned}
|I_3|&\leq C \int_{R_1}^{R_2}\left(\int_{\mathbb{T}^{d-1}} |f_2(y_1,y')|^2 dy'\right)^{1/2}\cdot
\left(\int_{\mathbb{T}^{d-1}} |\nabla_{y'}\tilde{N}_2(y_1,y')|^2 dy'\right)^{1/2}dy_1\\
& \leq C \left(\int_{(R_1,R_2)\times \mathbb{T}^{d-1}}|f_2|^2dy\right)^{1/2}\cdot
\left(\int_{(R_1,R_2)\times \mathbb{T}^{d-1}}|\nabla_y\tilde{N}_2|^2dy\right)^{1/2}\\
& \leq C \exp\{-CR_1\}\left(\int_{(R_1,R_2)\times \mathbb{T}^{d-1}}|\nabla_y\tilde{N}_2|^2dy\right)^{1/2}.
\end{aligned}\end{equation}

Therefore, due to $\nabla \tilde{N}_2\in  L^2({\mathbb{D}})$ and combining \eqref{3.10}-\eqref{3.14} after letting $R_2\rightarrow +\infty$ in \eqref{3.10}, we have
\begin{equation*}
\int_{(R_1,+\infty)\times \mathbb{T}^{d-1}}|\nabla \tilde{N}_2|^2dy \leq C \int_{\{R_1\}\times \mathbb{T}^{d-1}}|\nabla \tilde{N}_2|^2dy'+C \exp\{-CR_1\},
\end{equation*}
which, due to Gronwall' Lemma, further implies that
\begin{equation}\label{3.15}
\int_{(R_1,+\infty)\times \mathbb{T}^{d-1}}|\nabla \tilde{N}_2|^2dy \leq C\exp\{-\theta R_1\},
\end{equation}
for some constant $\theta>0$. Then, according to the Lipschitz regularity estimates, $|\nabla \tilde{N}_2|\leq C\exp\{-\theta |y_1|\}$, as $y_1\rightarrow \infty$. Similarly, we can obtain that $|\nabla \tilde{N}_2|\leq C\exp\{-\theta |y_1|\}$, as $y_1\rightarrow -\infty$.

Consequently, $N_1+\tilde{N}_2$ is the desired solution satisfying $\nabla (N_1+\tilde{N}_2)\in C^{0,\beta'}(\mathbb{D})$, for any $\beta'\in (0,1)$, and $\nabla (N_1+\tilde{N}_2)$ decays  exponentially fast in $y_1$. Thus, we have complete this proof.
\end{proof}

To introduce the following result, we denote
$$\mathcal{H}=\left\{u:\nabla u\in L^2(\mathbb{D}), \fint_{\mathbb{D}_1}u=0\right\},$$
which is a Hilbert space equipped with the inner product: $$\langle u,v\rangle_{\mathcal{H}}=:\int_{\mathbb{D}}\nabla u \cdot \nabla vdy.$$

\begin{lemma}\label{t3.2}
Let $f\in L^2(\mathbb{D})$ with $supp(f)\in \mathbb{D}_1$, then there exists a unique solution $u\in \mathcal{H}$ to the Poisson equation $\Delta u=f$ in $\mathbb{D}$, with the energy estimates $||\nabla u||_{L^2(\mathbb{D})}\leq C ||f||_{L^2(\mathbb{D})}$, for the constant $C$ depending only on $d$. Moreover,
we have the following decay estimates associated with $R_1$,
\begin{equation}\label{3.16}\begin{aligned}
&\int_{(R_1,+\infty)\times \mathbb{T}^{d-1}}|\nabla u|^2dy \leq C\exp\{-\theta |R_1|\},\ \text{ if }R_1>1;\\
&\int_{(-\infty,-R_1)\times \mathbb{T}^{d-1}}|\nabla u|^2dy \leq C\exp\{-\theta |R_1|\},\ \text{ if }R_1<-1,
\end{aligned}\end{equation}
for some constant $\theta$ depending only on $d$.
\end{lemma}
\begin{proof}
For any $v\in \mathcal{H}$, we define
\begin{equation*}\begin{aligned}
F(v)=:\int_{\mathbb{D}}fvdy=\int_{\mathbb{D}_1}fvdy=\int_{\mathbb{D}_1}f\left(v-\fint_{\mathbb{D}_1}v\right)dy
\end{aligned}\end{equation*}

which, by Poincar\'{e} inequality, implies that
\begin{equation*}\begin{aligned}
|F(v)|\leq C||f||_{L^2(\mathbb{D})}||\nabla v||_{L^2(\mathbb{D})}.
\end{aligned}\end{equation*}

Thus, $F$ is a bounded linear functional on $\mathcal{H}$. Then according to the calssical Lax-Milgram Theorem, there exists a unique solution $u\in \mathcal{H}$ to  the Poisson equation $\Delta u=f$ in $\mathbb{D}$, satisfying the energy estimates $||\nabla u||_{L^2(\mathbb{D})}\leq C ||f||_{L^2(\mathbb{D})}$.

To see the integral $\int_{(R_1,+\infty)\times \mathbb{T}^{d-1}}|\nabla u|^2dy$ decays exponentially fast in $R_1>1$, we first note that $u$ satisfies the equation $\Delta u=0$ if $y_1>1$. For $\infty>R_2>R_1>1$,
multiplying the equation $\Delta u=0$ by $u$ and integrating over $(R_1,R_2)\times \mathbb{T}^{d-1}$ yields that
\begin{equation}\label{3.17}\begin{aligned}
\int_{(R_1,R_2)\times \mathbb{T}^{d-1}}\nabla u\cdot \nabla u=&\int_{\{R_2\}\times \mathbb{T}^{d-1}}\partial_{y_1} u\cdot u-\int_{\{R_1\}\times \mathbb{T}^{d-1}}\partial_{y_1} u\cdot u.
\end{aligned}\end{equation}

Next, similar to \eqref{3.11}, for any $R>1$, there holds
\begin{equation*}
\int_{\{R\}\times \mathbb{T}^{d-1}}\partial_{y_1} u=0.
\end{equation*}
Then, totally same to the proof of \eqref{3.15}, we have
\begin{equation}\label{3.18}
\int_{(R_1,+\infty)\times \mathbb{T}^{d-1}}|\nabla u|^2dy \leq C\exp\{-\theta R_1\},
\end{equation}
for some constant $\theta>0$. The proof of $R_1<-1$ is totally similar to the proof of \eqref{3.18}.

Consequently, we complete this proof.
\end{proof}

\begin{lemma}\label{t3.3}
Under the conditions \eqref{1.2}-\eqref{1.3} and $\int_\mathbb{Y} \tilde{b}_{\pm,i}m_\pm dy=0$, for $i=1,\cdots,d$,
there exists a $\mathbb{D}$-periodic matrix function $\phi_b$, such that
\begin{equation}\label{3.19}\begin{aligned}
&b_i(y)=\partial_{k}\phi_{b, ki}(y),\ \phi_{b, ki}(y)=-\phi_{b, ik}(y),\ \phi_b\in L^\infty(\mathbb{D});\\
&\phi_{b,ki}\text{ decays exponentially fast to } q_+\phi_{+,ki} \text{ in }y_1>1;\\
&\phi_{b,ki}\text{ decays exponentially fast to } q_-\phi_{-,ki} \text{ in }y_1<-1,
 \end{aligned}\end{equation}
 with $b$ defined in \eqref{1.11},
for some 1-periodic $\phi_\pm$ satisfying $\phi_\pm=-\phi_\pm^*$ and $\phi_\pm\in L^\infty(\mathbb{R}^d)$.
\end{lemma}
\begin{proof}
Denote ${\beta}_{\pm,i}=\tilde{b}_{\pm,i}m_{\pm}$ and ${a}_{\pm,ij}=\tilde{a}_{ij}m_{\pm}$, then set
\begin{equation}\label{3.20}
{b}_{\pm,i}(y)={\beta}_{\pm,i}(y)-\partial_{y_j}{a}_{\pm,ij}(y),
\end{equation}
Then, it follows from \eqref{1.5} and $\int_\mathbb{Y} \tilde{b}_{\pm,i}m_\pm dy=0$, for $i=1,\cdots,d$, that
\begin{equation}\label{3.21}
\partial_{y_i}{b}_{\pm,i}=0 \text{ in }\mathbb{Y},\quad \int_\mathbb{Y}{b}_i(y)dy=0.
\end{equation}
It follows from \eqref{3.21} that there exist 1-periodic matrix functions $\phi_{\pm}$, such that

\begin{equation}\label{3.22}b_{\pm,i}(y)=\partial_{k}\phi_{\pm,ki}(y),\ \phi_{\pm,ki}(y)=-\phi_{\pm,ik}(y) \text{ in }\mathbb{Y},\quad \fint_\mathbb{Y} \phi_{\pm}=0, \end{equation}
respectively, which were obtained by first solving the Poisson equations
$$\Delta N_{\pm,i}={b}_{\pm,i} \quad \text{ in }\mathbb{Y},\quad \fint_\mathbb{Y} N_{\pm,i}=0,$$
and then setting
$$\phi_{\pm,ki}=\partial_{y_k}N_{\pm,i}-\partial_{y_i}N_{\pm,k}.$$

Similarly, we  want solve the Poisson equation
\begin{equation}\label{3.23}
\Delta N_{b,i}= b_i \quad \text{ in}\quad \mathbb{D}.
\end{equation}

Choosing cut-off functions $\psi_\pm(x_1)$, such that
\begin{equation}\label{3.24}\begin{aligned}
&\psi_+(x_1)=1,\text{ if }x_1\geq 1,\ \psi_+(x_1)=0,\text{ if }x_1\leq 0;\\
&\psi_-(x_1)=1,\text{ if }x_1\leq -1,\ \psi_-(x_1)=0,\text{ if }x_1\geq 0.
\end{aligned}\end{equation}
Then, the existence of the solution to the Poisson Equation \eqref{3.23} is equivalent to the following equation:
\begin{equation}\label{3.25}\begin{aligned}
&\Delta (N_{b,i}-q_+\psi_+N_{+,i}-q_-\psi_-N_{-,i})\\
=&b_i-q_+\psi_+\Delta N_{+,i}-q_-\psi_-\Delta N_{-,i}-2q_+\nabla\psi_+ \nabla N_{+,i}\\
&-2q_-\nabla\psi_- \nabla N_{-,i}-q_+\Delta \psi_+N_{+,i}-q_-\Delta\psi_-N_{-,i}\\
=&(1-\psi_+-\psi_-)b_i+\psi_+(b_i-q_+b_{+,i})+\psi_-(b_i-q_-b_{-,i})\\
&-2q_+\nabla\psi_+ \nabla N_{+,i}-2 q_-\nabla\psi_- \nabla N_{-,i}-q_+\Delta \psi_+N_{+,i}-q_-\Delta\psi_-N_{-,i}\\
=&:\tilde{G}_1+\tilde{G}_2+\tilde{G}_3,
\end{aligned}\end{equation}
for $\tilde{G}_2=:\psi_+(b_i-q_+b_{+,i})$ and $\tilde{G}_3=:\psi_-(b_i-q_-b_{-,i})$.
Note that
\begin{equation}\label{3.26}\begin{aligned}
\psi_\pm(b_i-q_\pm b_{\pm,i})
=&\psi_\pm\left({\beta}_{i}-\partial_{y_j}{a}_{ij}-q_{\pm}{\beta}_{\pm,i}
+q_{\pm}\partial_{y_j}{a}_{\pm,ij}\right)\\
=&\psi_\pm\left[\tilde{b}_{\pm,i}(m-q_\pm m_\pm)-\partial_{y_i}(\tilde{a}_{ij}(m-q_\pm m_\pm))\right],
\end{aligned}\end{equation}
where we have used \eqref{1.11} and \eqref{3.20}. Due to \eqref{3.2} and \eqref{3.4}, $\tilde{G}_2$ decays  exponentially fast in $y_1$. Moreover, we may assume
$b_i-q_+b_{+,i}\not\equiv 0$ on $[0,1]\times \mathbb{T}^{d-1}$, for otherwise, we can define $m\equiv q_+
m_+$ in $[1,\infty)\times \mathbb{T}^{d-1}$, which implies $b_i-q_+b_{+,i}\equiv 0$ in
$[0,+\infty)\times \mathbb{T}^{d-1}$, due to \eqref{3.26} and $\tilde{b}(y)=\tilde{b}_+(y)$ if $y_1\geq 1$. Thus $b_i$ is smooth in $[0,+\infty)\times \mathbb{T}^{d-1}$, which implies $m$ is smooth in $[0,+\infty)\times \mathbb{T}^{d-1}$ after in view of \eqref{1.9} and \eqref{1.11}, since $\tilde{b}$ and $\tilde{a}$ are smooth.
Due to $\partial_i b_i=q_+\partial_i b_{+,i}=0$ in $[0,+\infty)\times \mathbb{T}^{d-1}$, we know that this $m$ we just defined  solves the
equation $\partial_{y_iy_j}\left(\tilde{a}_{ij}(y)
m(y)\right)-\partial_{y_i}\left(\tilde{b}_{i}(y)m(y)\right)=0$ in $[0,+\infty)\times \mathbb{T}^{d-1}$ with
$\text {supp}(\psi_+(b_i-q_+b_{+,i}))\subset[0,1]\times \mathbb{T}^{d-1}$. According to the uniqueness of the invariant measure defined in \eqref{1.8}, the $m$ we constructed is the invariant measure for $\tilde{\mathcal{L}}$, satisfying $m-q_+
m_+\equiv0$ in $[1,\infty)\times \mathbb{T}^{d-1}$ and $\text {supp}(\psi_+(b_i-q_+b_{+,i}))\subset[0,1]\times \mathbb{T}^{d-1}$ (actually, in the 1-D case, $m-q_+
m_+\equiv0$ in $[1,\infty)\times \mathbb{T}^{d-1}$). Therefore, we can add the term
$\psi_+(b_i-q_+b_{+,i})$ into the discussion of $\tilde{G}_1$.

To proceed, due to \eqref{3.26}, we first note that the integral $\int_{[1,\infty)\times \mathbb{T}^{d-1}}|b_i-q_+b_{+,i}|$
exists and equals to some constant, and $b_i-q_+b_{+,i}\not\equiv0$ is smooth on the integral $[0,1]\times \mathbb{T}^{d-1}$, then we can
choose a smooth function $\psi_+(x_1)$, satisfying
\begin{equation}\label{3.27}\begin{aligned}
&\psi_+(x_1)=1,\text{ if }x_1\geq 1,\ \psi_+(x_1)=0,\text{ if }x_1\leq 0;\\
&\quad \int_{(0,\infty)\times\mathbb{T}^{d-1}}\psi_+(x_1)(b_i-q_+b_{+,i})(x)=0.
\end{aligned}\end{equation}
Note that we do not assume that $\psi_+\geq 0$. Similarly, we may assume
$b_i-q_-b_{-,i}\not\equiv 0$ on $[-1,0]\times \mathbb{T}^{d-1}$, and choose a smooth function $\psi_-(x_1)$, satisfying
\begin{equation}\label{3.28}\begin{aligned}
&\psi_-(x_1)=1,\text{ if }x_1\leq -1,\ \psi_-(x_1)=0,\text{ if }x_1\geq 0;\\
&\quad \int_{(-\infty,0)\times\mathbb{T}^{d-1}}\psi_-(x_1)(b_i-q_-b_{-,i})(x)=0.
\end{aligned}\end{equation}

Since $\tilde{G}_1\in L^\infty(\mathbb{D})$ with $\text{supp}(\tilde{G}_1)\subset \mathbb{D}_1$, then Lemma \ref{t3.2} and the Lipschitz regularity estimates ensure that there exists a solution $u_1$ to the Poisson equation $\Delta u_1=\tilde{G}_1$ in $\mathbb{D}$ with the property that $ \nabla u_1$ decays  exponentially fast in $y_1$.

Moreover, according to Lemma \ref{t3.1}, there exists a solution  $u_2$ to the Poisson equation $\Delta u_2=\tilde{G}_2+\tilde{G}_3$ in $\mathbb{D}$ with the property that $ \nabla u_2$ decays  exponentially fast in $y_1$.

Therefore, there exists a solution $N_{b,i}$ to the Poisson equation $\Delta N_{b,i}=b_i$ in $\mathbb{D}$, such that $\nabla (N_{b,i}-q_+\psi_+N_{+,i}-q_-\psi_-N_{-,i})$  decays  exponentially fast in $y_1$.

Since $\Delta_y\partial_{y_i}N_{b,i}=\partial_{y_i}b_{i}=0$ in $\mathbb{D}$, $N_{b,i}$ is periodic in $y'$, and $||\nabla_y N_{b,i}||_{L^\infty(\mathbb{D})}\leq C$,  we know that $\partial_{y_i}N_{b,i}$ is a bounded harmonic function in $\mathbb{R}^d$ and thus $\partial_{y_i}N_{b,i}$ is a constant.

Let
\begin{equation*}
\phi_{b,ki}=:\partial_{y_k}N_{b,i}-\partial_{y_i}N_{b,k}.
\end{equation*}
then
\begin{equation*}
\partial_{y_k}\phi_{b,ki}=\Delta_yN_{b,i}-\partial_{y_i}\partial_{y_k}N_{b,k}=b_{i}\quad \text{ in }\mathbb{D},
\end{equation*}
and \begin{equation*}\begin{aligned}
&\phi_{b,ki}\text{ decays exponentially fast to } q_+\phi_{+,ki} \text{ in }y_1>1;\\
&\phi_{b,ki}\text{ decays exponentially fast to } q_-\phi_{-,ki} \text{ in }y_1<-1.
\end{aligned}\end{equation*}
Consequently, we complete this proof.
\end{proof}

After obtaining the above results, we are ready to transfer non-divergence form \eqref{1.1} (or \eqref{1.10}) into divergence form \eqref{2.1}. Actually, due to \eqref{1.10} and Lemma \ref{t3.3}, we have
\begin{equation}\label{3.29}\begin{aligned}\partial_i\left({a}_{ij}^\varepsilon\partial_{j}u_\varepsilon\right)
+\frac{1}{\varepsilon}{b}_i^\varepsilon
\partial_i u_\varepsilon=\partial_i\left(({a}_{ij}^\varepsilon+\phi_{b,ij}^\varepsilon)\partial_{j}u_\varepsilon\right)
=fm^\varepsilon\quad\text{in } \mathbb{R}^d,
\end{aligned}\end{equation}
where the coefficient matrix $A+\phi_b$ and the source term $m$ satisfy the following conditions:
\begin{equation}\label{3.30}\begin{aligned}
A,\phi_b,m \text{ are smooth in }\mathbb{D} \text{ and 1-periodic in }y';\ \phi_{b,ij}=-\phi_{b,ji},\ i,j=1,\cdots,d;\\
\Lambda |\xi|^2\geq (A+\phi_b)\xi\cdot\xi=A\xi\cdot\xi\geq \lambda|\xi|^2,\forall\xi\in\mathbb{R}^d,\exists \text{ constants }0<\lambda\leq \Lambda;\\
A+\phi_b,m \text{ decays exponentially fast to }q_+(A_++\phi_+),q_+m_+,\text{ respectively, }\text{if }y_1\rightarrow +\infty;\\
A+\phi_b,m \text{ decays exponentially fast to }q_-(A_-+\phi_-),q_-m_-,\text{ respectively, }\text{if }y_1\rightarrow -\infty,\\
\exists \text{ positive constants }q_\pm;\ A_\pm,\phi_\pm,m_\pm \text{ are 1-periodic and smooth in }y; \ \phi_{\pm}=-\phi_{\pm}^*,\\
\Lambda|\xi|^2\geq (A_\pm+\phi_\pm)\xi\cdot\xi=A_\pm\xi\cdot\xi\geq \lambda|\xi|^2,\forall\xi\in\mathbb{R}^d;\ \fint_{\mathbb{Y}}m_\pm=1.
\end{aligned}\end{equation}

Consequently, we have achieved the aim of this section, and in the next section, our effort is to investigate the elliptic equations of divergence form \eqref{3.29} with coefficient matrix and source term decaying exponentially fast to some different periodic structures across the interface.

\section{Quantitative results for the divergence form}
From now on, we investigate the following divergence-form in elliptic homogenization. Precisely, for $0<\varepsilon<1$, consider
\begin{equation}\label{4.1}
\left\{\begin{aligned}\mathcal{L}_\varepsilon u_\varepsilon=\partial_i\left({a}_{ij}^\varepsilon\partial_{j}u_\varepsilon\right)&
=fm^\varepsilon\quad \text{ in }\quad \mathbb{R}^d,\\
u_\varepsilon& \in \dot{H}^1(\mathbb{R}^d),
\end{aligned}\right.\end{equation}
where the coefficient matrix $A$ and the source term $m$ satisfy the following conditions:
\begin{equation}\label{4.2}\begin{aligned}
A,m \text{ are 1-periodic in }y',\ A,A_\pm\in C^{0,\alpha}(\mathbb{R}^d),\ ||m||_{L^\infty(\mathbb{R}^d)}\leq \Lambda;\\
\Lambda|\xi|^2\geq A\xi\cdot\xi\geq \lambda|\xi|^2,\forall\xi\in\mathbb{R}^d,\exists \text{ constants }0<\lambda\leq \Lambda;\\
A,m \text{ decays exponentially fast to }q_+A_+,q_+m_+,\text{ respectively, }\text{if }y_1\rightarrow +\infty;\\
A,m \text{ decays exponentially fast to }q_-A_-,q_-m_-,\text{ respectively, }\text{if }y_1\rightarrow -\infty,\\
\text{for some positive constants }q_\pm\leq \Lambda;\ A_\pm,m_\pm \text{ are 1-periodic in }y; \fint_{\mathbb{Y}}m_\pm=1, \\
||m_\pm||_{L^\infty(\mathbb{R}^d)}\leq \Lambda;\ \Lambda|\xi|^2\geq A_\pm\xi\cdot\xi\geq \lambda|\xi|^2,\forall\xi\in\mathbb{R}^d.\
\end{aligned}\end{equation}
Precisely, the decay property means that
\begin{equation*}A(y)=\left\{\begin{aligned}&A_{>}(y)\xrightarrow[\text{decays exponentially fast to}]{} q_+A_+, \quad \text{if}\quad y_1>1,\\
&A_<(y)\xrightarrow[\text{decays exponentially fast to}]{}q_-A_-,\quad \text{if}\quad y_1<-1,
\end{aligned}\right.\end{equation*}

Actually, it seems that $A_>(y)$ is only defined in $[1,\infty)\times \mathbb{T}^{d-1}$. However, after a
suitable extension, we can find a $C^{0,\alpha}$ elliptic $\tilde{A_>}$, such that $\tilde{A_>}(y)=A_>(y)$ if $y_1>1$; and $\tilde{A_>}(y)=q_+A_+$ if $y_1\leq 0$. Then $\tilde{A_>}$ satisfies the assumptions in Section 2 and decays exponentially fast to $q_+A_+$ for $|y_1|\rightarrow +\infty$. From now on, we identity $A_>(y)$ with
$\tilde{A_>}(y)$. The similar extension also holds for $A_<$.

In view of the results in Section 2, we may expect that the homogenized operator $\mathcal{L}_0=\operatorname{div}(\widehat{A}\nabla\cdot)$ for $\mathcal{L}_\varepsilon$ has the following form:
\begin{equation}\label{4.3}
\widehat{A}(x)=\left\{\begin{aligned}
q_+\widehat{A_+}=q_+\widehat{A_>},\quad\text{ if }\quad x_1>0,\\
q_-\widehat{A_-}=q_-\widehat{A_<},\quad\text{ if }\quad x_1<0,
\end{aligned}\right.\end{equation}
where $\widehat{A_\pm}$ are the homogenized matrices associated with the periodic matrices $A_\pm$. In general, the matrix $\widehat{A}$ is discontinuous across the interface $\mathcal{I}=:\{x:x_1=0\}$.

\subsection{$\widehat{A}$-harmonic functions}
Inspired by the works \cite{MR3421758,MR3974127}, we introduce the $\widehat{A}$-harmonic functions, which is spanned by the constants functions and the following piecewise linear functions:
\begin{equation}\label{4.4}
P_j(x)=P(x)\cdot e_j=:\left\{\begin{aligned}&x_j,&\text{ if } x_1<0,\\
&x_j+\vartheta_jx_1, &\text{ if } x_1>0,
\end{aligned}\right.\end{equation}
for $j=1,\cdots,d$, where $\vartheta=(\vartheta_1,\cdots,\vartheta_j)$ is related to the transmission matrix through the interface and defined as:
\begin{equation}\label{4.5}
\vartheta_j=\frac{\left(\widehat{a_-}\right)_{1j}-\left(\widehat{a_+}\right)_{1j}}{\left(\widehat{a_+}\right)_{11}}.
\end{equation}

If $\vartheta=0$, then $\widehat{A}$ is constant and the functions $P_j$ are linear.

It is straightforward to check that the functions $P_j$ are solution to
\begin{equation}\label{4.6}\operatorname{div}(\widehat{A}\nabla P_j)=0\quad \text{in}\quad \mathbb{R}^d.\end{equation}
 Actually, by definition, the functions $P_j$ are continuous and their gradients read as
\begin{equation}\label{4.7}\nabla P_j(x)=\left\{\begin{aligned}
&e_j, &\text{ if } x_1<0,\\
&e_j+\vartheta_je_1, &\text{ if } x_1>0.
\end{aligned}\right.\end{equation}

Hence, the functions $P_j$ are $\widehat{A}$-harmonic in $\mathbb{R}^*_-\times \mathbb{R}^{d-1}$ and in $\mathbb{R}^*_+\times \mathbb{R}^{d-1}$, and they satisfy the transmission conditions across the interface:
\begin{equation}\label{4.8}
\begin{gathered}
\lim _{h \rightarrow 0^{+}}\left[\left(\widehat{A} \cdot \nabla P_{j}\right)\left(x+h e_{1}\right)\right] \cdot e_{1}=\lim _{h \rightarrow 0^{+}}\left[\left(\widehat{A} \cdot \nabla P_{j}\right)\left(x-h e_{1}\right)\right] \cdot e_{1}, \\
\lim _{h \rightarrow 0^{+}} \partial_{k} P_{j}\left(x+h e_{1}\right)=\lim _{h \rightarrow 0^{+}} \partial_{k} P_{j}\left(x-h e_{1}\right),
\end{gathered}
\end{equation}
for all $x\in\mathcal{I}$ and $k=2,\cdots,d$.

\subsection{Basic estimates of the correctors}
Since the correctors are meant to turn the $\widehat{A}$-harmonic functions $P_j$ into $A$-harmonic sub-linear functions, the correctors $\chi_k$ should solve the following equation:
\begin{equation}\label{4.9}
\partial_{y_i}\left(a_{ij}(y)\partial_{y_j}(P_k(y)+\chi_k(y))\right)=0 \text{ in }\mathbb{D},\ \chi_k\text{ is 1-periodic in }y'.
\end{equation}

Moreover, it is known in Section 2 that the correctors $\chi_{>,k}$ (or $\chi_{<,k}$, respectively), $k=1,\cdots,d$, for the matrix $A_>$ (or $A_<$) are defined as:
\begin{equation}\label{4.10}
\partial_{y_i}\left(a_{>,ij}(y)\partial_{y_j}\chi_{>,k}(y)\right)=-\partial_{y_i}a_{>,ik}\quad \text{ in }\mathbb{D}.
\end{equation}
We refer to Section 2 for the precise definition and  properties of $\chi_{>,k}$ and $\chi_{<,k}$.

\begin{lemma}\label{t4.1}
Suppose that the matrix $A$ satisfies the conditions in \eqref{4.2}, then there exists a unique $\mathbb{D}$-periodic solution $\chi_j$ to Equation \eqref{4.9}, such that
\begin{equation}\label{4.11}\left\{\begin{aligned}
&\nabla (\chi_j-\chi_{<,j})\in L^2 ((0,+\infty)\times \mathbb{T}^{d-1}),\ |\chi_j-\chi_{<,j}|\rightarrow 0 \text{ as }y_1\rightarrow -\infty;\\
&\nabla (\chi_j-\chi_{>,j}-\vartheta_j \chi_{>,1})\in L^2 ((-\infty,0)\times \mathbb{T}^{d-1}),\ |\chi_j-\chi_{>,j}-\vartheta_j \chi_{>,1}|\rightarrow 0 \text{ as }y_1\rightarrow +\infty.
\end{aligned}\right.\end{equation}
 Moreover, there exist constants $C>0$ and $\theta>0$ such that
\begin{equation}\label{4.12}\left\{\begin{aligned}
|\nabla \chi_j(y)-\nabla \chi_{<,j}(y)|&\leq C\exp(-\theta|y_1|),\ \text{ if }y_1<-1,\\
|\nabla \chi_j(y)-\nabla \chi_{>,j}(y)-\vartheta_j \nabla&\chi_{>,1}(y)|\leq C\exp(-\theta|y_1|),\ \text{ if }y_1>1.
\end{aligned}\right.\end{equation}
\end{lemma}
\begin{proof}Actually, this proof is almost identical to \cite[Proposition 5.4]{MR3974127} and \cite[Theorem 5.1]{MR3421758}, and we provide it for completeness.
Choose cut-off functions $\psi_\pm(x_1)$, such that
\begin{equation}\label{4.13}\begin{aligned}
&\psi_+(x_1)=1,\text{ if }x_1\geq 1,\ \psi_+(x_1)=0,\text{ if }x_1\leq 0;\\
&\psi_-(x_1)=1,\text{ if }x_1\leq -1,\ \psi_-(x_1)=0,\text{ if }x_1\geq 0.
\end{aligned}\end{equation}

Next, we define
\begin{equation}\label{4.14}
v(y)=\chi_k(y)-\psi_+(y_1)\left(\chi_{>,k}(y)+\vartheta_k\chi_{>,1}(y)\right)-\psi_-(y_1)\chi_{<,k}(y).
\end{equation}

Therefore, according to \eqref{4.9},
\begin{equation}\label{4.15}\begin{aligned}
-\operatorname{div}(A\nabla v)&=\operatorname{div}(A\nabla P_k)+\operatorname{div}\left(A\cdot \nabla \left\{\psi_+\left(\chi_{>,k}+\vartheta_k\chi_{>,1}\right)+\psi_-\chi_{<,k}\right\}\right)\\
&=:\operatorname{div}(f)+\operatorname{div}(g),
\end{aligned}\end{equation}
where by adding the constant term $\widehat{A}\cdot \nabla P_k$ after noting \eqref{4.6},
$$f=:(1-\psi_+-\psi_-)(A-\widehat{A})\cdot\nabla P_k+A\cdot \nabla \psi_+\left(\chi_{>,k}+\vartheta_k\chi_{>,1}\right)+A\cdot\nabla\psi_-\chi_{<,k},$$
and using
\begin{equation}\label{4.16}\begin{aligned}
g=&:\psi_+\left(A(\nabla P_k+\nabla\chi_{>,k}+\vartheta_k \nabla \chi_{>,1})-\widehat{A_>}\nabla P_k\right)\\
&+\psi_-\left(A(\nabla P_k+\nabla\chi_{<,k})-\widehat{A_<}\nabla P_k\right)\\
=&:g_++g_-.
\end{aligned}\end{equation}

To proceed, we need to rewrite $\operatorname{g}$ in a more suitable form. For simplicity, we only calculate $g_+$. A direct computation shows that
$$\begin{aligned}\operatorname{div}{g_+}&=\partial_i\left(\psi_+\left(a_{ik}(\delta_{kl}+\partial_k\chi_{>,\ell})
-\widehat{a_{>}}_{i\ell}\right)\partial_\ell P_j\right)\\
&=-\partial_i\left(\psi_+\partial_k\phi_{>,ki\ell}\partial_\ell P_j\right)\\
&=-\partial_i\psi_+\partial_k\phi_{>,ki\ell}\partial_\ell P_j-\psi_+\partial_i\partial_k\phi_{>,ki\ell}\partial_\ell P_j\\
&=-\partial_k\left(\partial_i\psi_+\phi_{>,ki\ell}\partial_\ell P_j\right)+\partial_k\partial_i\psi_+\phi_{>,ki\ell}\partial_\ell P_j\\
&=-\partial_k\left(\partial_i\psi_+\phi_{>,ki\ell}\partial_\ell P_j\right),
\end{aligned}$$
since $\nabla P_j$ is constant everywhere but on the interface where $\psi_\pm$ vanish, $\phi_>$ is antisymmetry and $\partial_i\partial_k\phi_{>,ki\ell}=\partial_i B_{>,ij}=0$ given by Lemma \ref{t2.7}.

Thus there holds:
$$\operatorname{div}(g)=\operatorname{div}(\tilde{g})$$
for
$$\tilde{g}_k=-\partial_i\psi_+\phi_{>,ki\ell}\partial_\ell P_j-\partial_i\psi_-\phi_{<,ki\ell}\partial_\ell P_j.$$

Moreover, it is easy to see that
$$f,\tilde{g}\in L^\infty(\mathbb{D}),\ \text{supp}(f,g)\subset [-1,1]\times \mathbb{T}^{d-1}.$$

Consequently, it follows from the classical Lax-Milgram Theorem that there exists a unique $\mathbb{D}$-periodic solution $v$ to Equation \eqref{4.15} such that $\nabla v\in L^2(\mathbb{D})$ and $v\rightarrow 0$ as $|y_1|\rightarrow +\infty$. Equivalently, there exists a unique $\mathbb{D}$-periodic solution $\chi_j$ to the Equation \eqref{4.9}, satisfying \eqref{4.11}. Moreover, totally similar to the proof of \eqref{3.15} (or see \cite[Proposition 5.4]{MR3974127} for more details), we have
\begin{equation}\label{4.17}
\int_{(R_1,+\infty)\times \mathbb{T}^{d-1}}|\nabla v|^2dy \leq C\exp\{-\theta R_1\},
\end{equation}which, by $L^\infty$-estimates, we finally obtain the desired estimate \eqref{4.12}.
\end{proof}

\begin{lemma}\label{t4.2}
Under the conditions in Lemma \ref{t4.1}, $\chi\in L^\infty(\mathbb{D})$.\end{lemma}
\begin{proof}
The proof is identical to Lemma \ref{t2.8}, after noting \eqref{4.11}, \eqref{4.12} and \eqref{4.15}.
\end{proof}

\subsection{Effective equation}
To determine the effective equation, we first determine the strong convergence in $H^{-1}$ of the source terms in \eqref{4.1}, which is stated in the following two lemmas.

\begin{lemma}\label{t4.3}
Assume that $m$ satisfies the assumptions in \eqref{4.2}, then
there exists a $\mathbb{D}$-periodic solution $\varphi_m$ to the Poisson equation $\Delta_y \varphi_m=m-q_+{I}_{y_1>0}-q_-I_{y_1<0}$ in $\mathbb{D}$, such that $\nabla \varphi_m\in C^{0,\beta'}(\mathbb{D})$, for any $\beta'\in (0,1)$, where ${I}_{y_1>0}$ and $I_{y_1<0}$ are the characteristic functions.
\end{lemma}
\begin{proof}
It is known that there exist 1-periodic functions $\varphi_{m_\pm}\in W^{2,p}(\mathbb{Y})$, for any $1<p<\infty$, satisfying $\Delta_y \varphi_{m_\pm}=m_{\pm}-1$ in $\mathbb{Y}$ with $\fint_\mathbb{Y} \varphi_{m_\pm}=0$, due to $\fint_\mathbb{Y} m_\pm=1$. Choose cut-off functions $\psi_\pm(x_1)$, such that
\begin{equation}\label{4.18}\begin{aligned}
&\psi_+(x_1)=1,\text{ if }x_1\geq 1,\ \psi_+(x_1)=0,\text{ if }x_1\leq 0;\\
&\psi_-(x_1)=1,\text{ if }x_1\leq -1,\ \psi_-(x_1)=0,\text{ if }x_1\geq 0.
\end{aligned}\end{equation}
Then the solvability of the Poisson equation $\Delta_y \varphi_m=m-q_+{I}_{y_1>0}-q_-I_{y_1<0}$ in $\mathbb{D}$ is equivalent to
\begin{equation}\label{4.19}\begin{aligned}
&\Delta(\varphi_m-q_+\psi_+\varphi_{m_+}-q_-\psi_-\varphi_{m_-})\\
=&m-q_+{I}_{y_1>0}-q_-I_{y_1<0}-q_+\psi_+(m_+-1)-q_-\psi_-(m_--1)\\
&-q_+\Delta\psi_+m_+-2q_+\nabla\psi_+\nabla m_+-q_-\Delta\psi_-m_--2q_-\nabla\psi_-\nabla m_-\\
=&(1-\psi_+-\psi_-)m+\psi_+(m-q_+m_+)+\psi_-(m-q_-m_-)-q_+(I_{y_1>0}-\psi_+)\\
&-q_-(I_{y_1<0}-\psi_-)-q_+\Delta\psi_+m_+-2q_+\nabla\psi_+\nabla m_+-q_-\Delta\psi_-m_--2q_-\nabla\psi_-\nabla m_-\\
=&:\tilde{F}.
\end{aligned}\end{equation}

Due to \eqref{4.2} and \eqref{4.18}, it is easy to see that  $\tilde{F}\in L^1(\mathbb{D})\cap L^\infty(\mathbb{D})$, then according to Lemma \ref{t2.6}, there exists a $\mathbb{D}$-periodic solution $\varphi_m$ to the Poisson equation $\Delta_y \varphi_m=m-q_+{I}_{y_1>0}-q_-I_{y_1<0}$ in $\mathbb{D}$, such that $\nabla (\varphi_m-q_+\psi_+\varphi_{m_+}-q_-\psi_-\varphi_{m_-})\in C^{0,\beta'}(\mathbb{D})$, which implies $\nabla \varphi_m \in C^{0,\beta'}(\mathbb{D})$, for any $\beta'\in (0,1)$.
\end{proof}

To proceed, we can obtain the following strong convergence in $H^{-1}$ with the help of Lemma \ref{t4.3}.
\begin{lemma}\label{t4.4}
Assume that $m$ satisfies the conditions in \eqref{4.2}, and $f\in L^2(\Omega)$ with $\Omega$ being a bounded Lipschitz domain in $\mathbb{R}^d$ for $d\geq 2$, then $f(x)m^\varepsilon(x)\rightarrow f(x)(q_+{I}_{x_1>0}+q_-I_{x_1<0})(x)$ strongly in $H^{-1}(\Omega)$, as $\varepsilon\rightarrow 0$.
\end{lemma}

\begin{proof}
We first introduce the concept of $\varepsilon$-smoothing operator $S_\varepsilon$. Fix a nonnegative function $\rho\in C_0^\infty(B(0,1/2))$ such that $\int_{\mathbb{R}^n}\rho(x) dx=1$. For $\varepsilon>0,$
define \begin{equation}\label{4.20}
S_\varepsilon(f)(x)=(\rho_\varepsilon\ast f)(x)=\int_{\mathbb{R}^n}f(x-y)\rho_\varepsilon(y)dy,\end{equation}
where $\rho_\varepsilon(y)=\varepsilon^{-d}\rho(y/\varepsilon)$. Moreover, we have the following estimates (for the proof, see \cite[Proposition 3.1.6]{shen2018periodic} and \cite[Lemma 3.1]{MR3596717} for example).
\begin{equation}\label{4.21}\begin{aligned}
&||S_{\varepsilon^{1/2}}(g)||_{L^2(\Omega)}\leq C ||g||_{L^2(\Omega)},\\
&||S_{\varepsilon^{1/2}}(g)-g||_{L^2(\Omega)}\rightarrow 0 \text{ as }\varepsilon \rightarrow 0, \text{ if }||g||_{L^2(\Omega)}\leq C,\\
&\varepsilon^{1/2}||\nabla S_{\varepsilon^{1/2}}(g)||_{L^2(\Omega)} \leq C ||g||_{L^2(\Omega)}.
\end{aligned}\end{equation}
Then, according to Lemma \ref{t4.3}, it is easy to see that
\begin{equation}\label{4.22}\begin{aligned}
&f  m^\varepsilon- f(q_+{I}_{x_1>0}+q_-I_{x_1<0})\\
=&\left(f-S_{\varepsilon^{1/2}}(f)\right)(m(x/\varepsilon)-(q_+{I}_{x_1>0}+q_-I_{x_1<0}))+\varepsilon^2
S_{\varepsilon^{1/2}}(f) \Delta_x\varphi_m^{\varepsilon}\\
=&\left(f-S_{\varepsilon^{1/2}}(f)\right)(m(x/\varepsilon)-(q_+{I}_{x_1>0}+q_-I_{x_1<0}))
+\varepsilon\operatorname{div}\left(S_{\varepsilon^{1/2}}(f)\nabla_y \varphi_m^{\varepsilon}\right)\\
&-\varepsilon \nabla  S_{\varepsilon^{1/2}}(f)\nabla_y \varphi_m^{\varepsilon}.
\end{aligned}\end{equation}
Consequently,
\begin{equation*}\begin{aligned}
&||f m^\varepsilon- f(q_+{I}_{x_1>0}+q_-I_{x_1<0})||_{H^{-1}(\Omega)}\\
\leq & C
||f-S_{\varepsilon^{1/2}}(f)||_{L^2(\Omega)}+C\varepsilon
||S_{\varepsilon^{1/2}}(f)||_{L^2(\Omega)}+C\varepsilon ||\nabla  S_{\varepsilon^{1/2}}(f)||_{L^2(\Omega)}\\
\rightarrow& \  0\quad \text{ as }\varepsilon\rightarrow 0,
\end{aligned}\end{equation*}after noting \eqref{4.21}.
Thus we complete this proof.
\end{proof}

Equipped with the correctors obtained in Lemma \ref{t4.1}, we are ready to state the following uniform $H$-convergence, which is similar to the proof of  Lemma \ref{2.3}.

\begin{lemma}\label{t4.5}
Suppose that the matrix $A$ satisfies the conditions in \eqref{4.2}. Let sequences $x_\ell\in \mathbb{R}^d$ and $\varepsilon_\ell\in \mathbb{R}_+$ satisfy $x_\ell \cdot e_1\rightarrow x_{0,1}\in \mathbb{R}$ and $\varepsilon_\ell\rightarrow 0$. Then, the sequence $A_\ell((\cdot-x_\ell)/\varepsilon_\ell)$
$H$-convergent to $\widehat{A}(\cdot-x_{0,1}e_1)$ on every Lipschipz bounded domain of $\mathbb{R}^d$.\end{lemma}

\begin{proof}
The proof is classical and relies on the Div-Curl Lemma \cite[Theorem 2.3.1]{shen2018periodic}. Therefore, we only emphasize on its main ingredient: the matrix $A$ admits correctors $\chi_k$ such that
\begin{equation}\label{4.23}\nabla \chi_k\in L^2_{\text{unif}}(\mathbb{R}^d,\mathbb{R}^d),\end{equation}
and, for any bounded Lipschitz domain $\Omega$, that satisfy the following weak convergence in $L^2(\Omega,\mathbb{R}^d)$,
\begin{equation}\label{4.24}\begin{aligned}
\nabla_y \chi_k((\cdot-x_\ell)/\varepsilon_\ell)\rightharpoonup 0 \text{ weakly in }L^2(\Omega,\mathbb{R}^d),
\end{aligned}\end{equation}
\begin{equation}\label{4.25}\begin{aligned}
(A\cdot(\nabla  P_k+\nabla \chi_k))((\cdot-x_\ell)/\varepsilon_\ell)
\rightharpoonup \widehat{A}\nabla P_k\text{ weakly in }L^2(\Omega,\mathbb{R}^d).
\end{aligned}\end{equation}
The above facts are consequences of \eqref{4.12}, using the properties \eqref{2.24} and \eqref{2.25} of the $\mathbb{D}$-periodic correctors $\chi_<$ and $\chi_>$. Moreover, see \eqref{2.24} and \eqref{2.25} for the similar proof of \eqref{4.24} and \eqref{4.25}, respectively.
\end{proof}

\begin{rmk}\label{t4.6}
Actually, it is easy to see that Theorem \ref{t1.1} follows readily from Lemma \ref{t4.4} and Lemma \ref{t4.5}.\end{rmk}

\subsection{Convergence rates}
In order to obtain the convergence rates, we first introduce some useful estimates obtained in \cite{MR3974127}.

Denote \begin{equation}\label{4.26}U_0(x)=:(\nabla P(x))^{-1}\cdot \nabla u_0.\end{equation}

By the transmission conditions \eqref{4.8} through the interface, the function $U_{0,j}$ is continuous across the interface $\mathcal{I}$, if $f$ is sufficiently regular, then $\widehat{A}\nabla u_0=A\nabla P \cdot (\nabla P(x))^{-1}\cdot \nabla u_0$ is regular, which implies $U_{0,j}=(\nabla P(x))^{-1}\cdot \nabla u_0$ is continuous across the interface $\mathcal{I}$.


 Recall that in the periodic case, we need $||u_0||_{H^2}$ to dominate the $L^2$ convergence rates of $u_\varepsilon-u_0$. But in our setting, due to $u_0\notin H^2$, we need to find a suitable adaption of $u_0$, such that this suitable adaption can dominate the $L^2$ convergence rates of $u_\varepsilon-u_0$. The following result states that $U_0$ is a suitable choice.

\begin{lemma}\label{t4.7}
Let $d\geq 3$, $x_0\in \mathbb{R}^d$, and $\widehat{A}$ be a matrix defined in \eqref{1.13} (or \eqref{4.3}). Suppose that $f\in L^{2d/(d+4)}(\mathbb{R}^d)\cap L^p(\mathbb{R}^d)$ for some $p\in(d,\infty)$. Let $u_0\in\dot{H}^1(\mathbb{R}^d)$ be a weak solution to $\mathcal{L}_0u_0=f$ in $\mathbb{R}^d$ and define $U_0$ by \eqref{4.26}. Then there exists a constant $C>0$, depending only on $d$, $\widehat{A_>}$ and $\widehat{A_<}$ such that
\begin{equation}\label{4.27}
||U_0||_{H^{1}(\mathbb{R}^d)}\leq C||f||_{L^{2d/(d+4)}(\mathbb{R}^d)\cap L^2(\mathbb{R}^d)}.
\end{equation}

Moreover,
\begin{equation}\label{4.28}
||U_0||_{W^{1,p}(\mathbb{R}^d)}\leq C||f||_{L^{2d/(d+4)}(\mathbb{R}^d)\cap L^p(\mathbb{R}^d)}.
\end{equation}

\end{lemma}

\begin{proof}
This proof is almost identical to \cite[Lemma 5.2]{MR3974127}, and we provide it for completeness. We first show an $L^2$ estimate on $u_0$. By definition, there holds:
$$u_0(x)=\int_{\mathbb{R}^d}\mathcal{G}_0(x,y)f(y)dy,$$
where  $\mathcal{G}_0$ is the Green function associated with the operator $\operatorname{div}(\widehat{A}\cdot \nabla)$ such that $|\mathcal{G}_0(x,y)|\leq C |x-y|^{2-d}$. Then applying the Hardy-Littlewood-Sobolev inequality yields that
%

\begin{equation}\label{4.29}
||u_0||_{L^2(\mathbb{R}^d)}\leq C ||f||_{L^{2d/(d+4)}(\mathbb{R}^d)}.
\end{equation}

To proceed, the function $\tilde{u}_0(x)=:u_0(P^{-1}(x))$ satisfies the following elliptic equation:

\begin{equation}\label{4.30}-\operatorname{div}\left(|J(x)|^{-1} \tilde{A}(x) \cdot \nabla \tilde{u}(x)\right)=|J(x)|^{-1} f\left(P^{-1}(x)\right)\quad\text{in}\quad \mathbb{R}^d,\end{equation}
where $\tilde{A}(x)$ is defined by
$$\tilde{A}(x):=\left(\nabla P\left(P^{-1}(x)\right)\right)^{T} \cdot \widehat{A}\left(P^{-1}(x)\right) \cdot \nabla P\left(P^{-1}(x)\right),$$
and $J(x)$ is the Jacobian of $P$ evaluated on $P^{-1}(x)$. By construction, $\tilde{A}(x)$ is elliptic and constant on the half-spaces $\mathbb{R}_\pm^*\times \mathbb{R}^{d-1}$, and the product $|J(x)|^{-1} \tilde{A}(x)$ is divergence-free in $\mathbb{R}^d$ due to \eqref{4.6}. Therefore, we can rewrite \eqref{4.30} as
$$\tilde{A}_{ij}\partial_{ij}\tilde{u}(x)=f(P^{-1}(x)).$$

Moreover, it follows from that \cite[Lemma 2.4]{MR2276531}, there exists a constant $C$ such that
\begin{equation}\label{4.31}
||\tilde{u}||_{H^2(\mathbb{R}^d)}\leq C ||\tilde{u}||_{L^2(\mathbb{R}^d)}+C||f||_{L^2(\mathbb{R}^d)}\leq
C||f||_{L^{2d/(d+4)}(\mathbb{R}^d)\cap L^2(\mathbb{R}^d)}.
\end{equation}
Therefore, a simple change of variable yields the desired estimate \eqref{4.27} after noting that $\partial_{x_j}\tilde{u}(x)=U_{0,j}(P^{-1}(x))$.

Similarly, applying the Hardy-Littlewood-Sobolev inequality also yields that
\begin{equation}\label{4.32}
||u_0||_{L^p(\mathbb{R}^d)}\leq C ||f||_{L^{pd/(d+2p)}(\mathbb{R}^d)},
\end{equation}
and applying \cite[Lemma 2.4]{MR2276531} again, we have
\begin{equation}\label{4.33}\begin{aligned}
||\tilde{u}||_{W^{2,p}(\mathbb{R}^d)}\leq& C ||\tilde{u}||_{L^p(\mathbb{R}^d)}+C||f||_{L^p(\mathbb{R}^d)}\\
\leq&C||f||_{L^{pd/(d+2p)}(\mathbb{R}^d)\cap L^p(\mathbb{R}^d)}\\
\leq& C ||f||_{L^{2d/(d+4)}(\mathbb{R}^d)\cap L^p(\mathbb{R}^d)},
\end{aligned}\end{equation}
where we have used $p>pd/(d+2p)>2d/(d+4)$ due to $p>d\geq 3$. Consequently, a simple change of variable yields the desired estimate \eqref{4.28} after noting that $\partial_{x_j}\tilde{u}(x)=U_{0,j}(P^{-1}(x))$.
\end{proof}

\noindent In the following lemma, we can define the so-called flux corrector $\phi$ associated with $\mathcal{L}$.

\begin{lemma}[Flux corrector]\label{t4.8}
 Under the condition \eqref{4.2}, denote the $\mathbb{D}$-periodic function
\begin{equation}\label{4.34}B_{ij}=:\widehat{A}_{ik}\partial_k P_j-A_{ik}(\partial_k P_j+\partial_k\chi_j),\end{equation}
for $i,j=1,\cdots,d$, then there exist the so-called flux corrector $\phi$, such that
\begin{equation*}\begin{aligned}
B_{ij}=\partial_{y_k}\phi_{kij}\text{ in }\mathbb{D},\ \phi_{kij}=-\phi_{ikj},\ \phi_{kij}\in L^\infty(\mathbb{D}),
\end{aligned}\end{equation*}
for $k,i,j=1,\cdots,d$.
\end{lemma}
\begin{proof}
We consider the Poisson equation $\Delta N_{ij}=B_{ij}$ in $\mathbb{D}$. It is known that in Lemma \ref{t2.7} that
\begin{equation}\label{4.35}\begin{aligned}\Delta N_{>,ij}=\widehat{A}_{>,ij}-A_{>,ij}+A_{>,ik}\partial_k\chi_{>,j},\\
\Delta N_{<,ij}=\widehat{A}_{<,ij}-A_{<,ij}+A_{<,ik}\partial_k\chi_{<,j}.
\end{aligned}\end{equation}

Then, we proceed in the same manner as in the proof of Lemma \ref{t4.3} by using techniques in \cite{MR3421758}. We decompose

\begin{equation*}N=\psi_+N_>\cdot \nabla P+\psi_-N_<\cdot \nabla P+\tilde{N},\end{equation*}
with the same $\psi_{\pm}$ defined in \eqref{4.18}.
Recall that $\nabla P$ is piecewise constant and possibly discontinuous only across the interface, where $\psi_\pm$ vanishes. Hence, by definition
\begin{equation}\label{4.36}
\begin{aligned}
\Delta \tilde{N}=& \Delta N-\psi_{+} \Delta N_{>} \cdot \nabla P-\psi_{-} \Delta N_{<} \cdot \nabla P \\
&-2\left(\nabla \psi_{+} \cdot \nabla N_{>} \cdot \nabla P+\nabla \psi_{-} \cdot \nabla N_{<} \cdot \nabla P\right) \\
&-\Delta \psi_{+} N_{>} \cdot \nabla P-\Delta \psi_{-} N_{<} \cdot \nabla P .
\end{aligned}
\end{equation}

Using \eqref{4.34} and \eqref{4.35} yields that

\begin{equation}\label{4.37}
\begin{aligned}
&\Delta  N_{i j}-\psi_{+} \Delta\left(\left(N_{>}\right)_{i k}\right) \partial_{k} P_{j}-\psi_{-} \Delta\left(\left(N_{<}\right)_{i k}\right) \partial_{k} P_{j} \\
=&\left(1-\psi_{+}-\psi_{-}\right)\left(\widehat{A}_{i k} \partial_{k} P_{j}-A_{i k}\left(\partial_{k} P_{j}+\partial_{k} \chi_{j}\right)\right) \\
&+\psi_{+} A_{i k}\left(\partial_{k} \chi_{>,\ell} \partial_{\ell} P_{j}-\partial_{k} \chi_{j}\right)+\psi_{-} A_{i k}\left(\partial_{k} \chi_{<,\ell} \partial_{\ell} P_{j}-\partial_{k} \chi_{j}\right).
\end{aligned}
\end{equation}

According to \eqref{4.12}, the right-hand term of \eqref{4.37} is $\mathbb{D}$-periodic and in $L^1(\mathbb{D})\cap L^\infty(\mathbb{D})$, so right-hand term of \eqref{4.36} is $\mathbb{D}$-periodic and in $L^1(\mathbb{D})\cap L^\infty(\mathbb{D})$, Then it follows from Lemma \ref{t2.6}, that there exists a solution $N_{ij}$ to the Poisson equation $\Delta N_{ij}=B_{ij}=\widehat{A}_{ik}\partial_k P_j-A_{ik}(\partial_k P_j+\partial_k\chi_j)$ in $\mathbb{D}$, such that $\nabla N\in L^\infty(\mathbb{D})$, after in view of Lemma \ref{t2.7}.

Moreover, due to \eqref{4.6} and \eqref{4.9}, $\Delta_y\partial_{y_i}N_{ij}=0$ in $\mathbb{D}$. Since $N_{ij}$ is $\mathbb{D}$-periodic, then $\partial_{y_i}N_{ij}$ is a bounded harmonic function in $\mathbb{R}^d$ and thus $\partial_{y_i}N_{ij}$ is a constant. Let
\begin{equation*}
\phi_{kij}=:\partial_{y_k}N_{ij}-\partial_{y_i}N_{kj}.
\end{equation*}
then
\begin{equation*}
\partial_{y_k}\phi_{kij}=\Delta_yN_{ij}-\partial_{y_i}\partial_{y_k}N_{kj}=B_{ij}\text{ in }\mathbb{D},
\end{equation*}
and \begin{equation*}
\phi_{kij}\in L^\infty(\mathbb{D}).
\end{equation*}
Consequently, we complete this proof.
\end{proof}

After obtaining the flux corrector $\phi$, we can state the following convergence rates.

\begin{thm}[Convergence rates I]\label{t4.9}
 Under the condition \eqref{4.2}, let $f\in L^{2d/(d+2)}(\mathbb{R}^d)\cap L^2(\mathbb{R}^d)$ as well as $\nabla f\in L^{2d/(d+2)}(\mathbb{R}^d)$ with $d\geq 3$, and denote
\begin{equation}\label{4.38}
w_\varepsilon=u_\varepsilon-u_0-\varepsilon\chi_j^\varepsilon S_\varepsilon(U_{0,j}),\end{equation}
with $S_\varepsilon$ defined by \eqref{4.20}, $u_\varepsilon$ defined by \eqref{4.1}, $u_0$ defined as
\begin{equation}\label{4.39}\left\{\begin{aligned}
{\mathcal{L}}_0 u_0= &f(x)(q_+{I}_{x_1>0}+q_-I_{x_1<0})(x)\quad \text{ in }\mathbb{R}^d,\\
&u_0\in \dot{H}^1(\mathbb{R}^d),
\end{aligned}\right.\end{equation} and $U_0$ defined by \eqref{4.26}.
Then there holds the following convergence rates:
\begin{equation}\label{4.40}|| w_\varepsilon||_{\dot{H}^1(\mathbb{R}^d)}\leq C\varepsilon\left(||f||_{L^{2d/(d+4)}(\mathbb{R}^d)\cap L^2(\mathbb{R}^d)}+||\nabla f||_{L^{2d/(d+2)}(\mathbb{R}^d)}\right),\end{equation}
where $C$ depends only on $d$, $\lambda$ and $\Lambda$.
\end{thm}
\begin{proof}
According to \eqref{4.1}, \eqref{4.12}, \eqref{4.27}, \eqref{4.39} and Lemma \ref{t4.2}, it is easy to see that $w_\varepsilon\in \dot{H}^1(\mathbb{R}^d)$. Moreover, a direct computation yields that
\begin{equation}\label{4.41}\begin{aligned}
\operatorname{div}(A^\varepsilon \nabla w_\varepsilon)
=&fm^\varepsilon-f(q_+{I}_{x_1>0}+q_-I_{x_1<0})+
\operatorname{div}\left((\widehat{A}-A^\varepsilon)\nabla u_0\right)\\
&-\varepsilon \operatorname{div}(A^\varepsilon\chi_j^\varepsilon S_\varepsilon(\nabla U_{0,j}))
-\operatorname{div}(A^\varepsilon\nabla_y\chi_j^\varepsilon S_\varepsilon(U_{0,j}))\\
=&f(m^\varepsilon-q_+{I}_{x_1>0}+q_-I_{x_1<0})+\operatorname{div}\left((\widehat{A}-A^\varepsilon)\cdot \nabla P\cdot (U_0-S_\varepsilon(U_{0}))\right)\\
&+\operatorname{div}\left(\left[(\widehat{A}-A^\varepsilon)\cdot \nabla P-A^\varepsilon\nabla_y\chi_j^\varepsilon\right] S_\varepsilon(U_{0})\right)-\varepsilon \operatorname{div}(A^\varepsilon\chi_j^\varepsilon S_\varepsilon(\nabla U_{0,j}))\\
=&\varepsilon^2 \Delta \varphi_m^\varepsilon\cdot f +\varepsilon \partial_k(\phi_{kij}^\varepsilon S_\varepsilon(\partial_i U_{0,j}))-\varepsilon \operatorname{div}(A^\varepsilon\chi_j^\varepsilon S_\varepsilon(\nabla U_{0,j}))\\
&+\operatorname{div}\left((\widehat{A}-A^\varepsilon)\cdot \nabla P(U_0-S_\varepsilon(U_{0}))\right)\\
=&\varepsilon \operatorname{div}_x(\nabla_y \varphi_m^\varepsilon\cdot f)-\varepsilon \nabla_y \varphi_m^\varepsilon\cdot \nabla f+\operatorname{div}\left((\widehat{A}-A^\varepsilon)\cdot \nabla P\cdot(U_0-S_\varepsilon(U_{0}))\right)\\
&+\varepsilon \partial_k(\phi_{kij}^\varepsilon S_\varepsilon(\partial_i U_{0,j}))-\varepsilon \operatorname{div}(A^\varepsilon\chi_j^\varepsilon S_\varepsilon(\nabla U_{0,j})),
\end{aligned}\end{equation}
where we have used Lemma \ref{t4.3} and Lemma \ref{t4.8} in \eqref{4.41}. Then, testing \eqref{4.41} by $w_\varepsilon$ yields that
\begin{equation*}\begin{aligned}
||\nabla w_\varepsilon||^2_{L^2(\mathbb{R}^d)}\leq& C\varepsilon\int_{\mathbb{R}^d}|f|\cdot|\nabla w_\varepsilon|+C\int_{\mathbb{R}^d}| U_0-S_\varepsilon(U_{0})|\cdot|\nabla w_\varepsilon|\\
&+C\varepsilon\int_{\mathbb{R}^d}|\nabla f|\cdot|w_\varepsilon|
+C\varepsilon\int_{\mathbb{R}^d}|S_\varepsilon(\nabla U_{0})|\cdot |\nabla w_\varepsilon|\\
\leq &C\varepsilon\left(||f||_{L^2(\mathbb{R}^d)}+||\nabla U_0||_{L^2(\mathbb{R}^d)}\right)||\nabla w_\varepsilon||_{L^2(\mathbb{R}^d)}\\
&+C \varepsilon ||\nabla f||_{L^{2d/(d+2)}(\mathbb{R}^d)}|| w_\varepsilon||_{L^{2d/(d-2)}(\mathbb{R}^d)},
\end{aligned}\end{equation*}
where we have used the following estimates
\begin{equation}\label{4.42}
||U_0-S_\varepsilon(U_{0})||_{L^2(\mathbb{R}^d)}\leq C\varepsilon ||\nabla U_0||_{L^2(\mathbb{R}^d)}
\text{ and }||S_\varepsilon(\nabla U_{0})||_{L^2(\mathbb{R}^d)}\leq C||\nabla U_{0}||_{L^2(\mathbb{R}^d)}.
\end{equation}
We refer to \cite[Proposition 3.1.6]{shen2018periodic} for the proof of $(4.42)$.

Consequently, using \eqref{4.27} and $||w_\varepsilon||_{L^{2d/(d-2)}(\mathbb{R}^d)}\leq C(d)||\nabla w_\varepsilon||_{L^2(\mathbb{R}^d)}$ yields the desired estimate \eqref{4.40}.
\end{proof}

The following result is a generation of the interior Lipschitz estimates \cite[Theorem 4.1.2]{shen2018periodic} for the periodic case.
\begin{thm}[Interior Lipschitz estimates]\label{t4.10}
 Suppose that $d\geq 2$ and the matrix $A$ satisfies the conditions in \eqref{4.2}. Let $\varepsilon>0$, $x_0\in \mathbb{R}^d$, $R>0$ and $B=B(x_0,2R)$. Assume that $u_\varepsilon\in H^1(B)$ is a weak solution to
$$\operatorname{div}(A(x/\varepsilon)\nabla u_\varepsilon)=0\quad \text{in}\quad B.$$
Then, there exists a constant $C$, depending only on $A$ and $d$ such that
$$||\nabla u_\varepsilon||_{L^\infty(B(x_0,R))}\leq CR^{-1}\left(\fint_{B(x_0,2R)}|u_\varepsilon|^2\right)^{1/2}.$$
\end{thm}
\begin{proof}The proof is almost identical to the proof of \cite[Theorem 4.1]{MR3974127}, since the effective equation in the sense of $H$-compactness is same as the case in \cite{MR3974127}, except for the local smoothness across the interface $\mathcal{I}$.

The proof of Theorem \ref{t4.10} is based on the method of compactness argument and it is done in the following three steps:\\

\noindent Step 1. [One-step improvement.] We take advantage of the uniform H-convergence of the multi-scale problem $\mathcal{L}_{\varepsilon} u_{\varepsilon}=0$ in $B$ to the homogeneous effective problem $\mathcal{L}_{0} u_{0}=0$ in $B$, which states that the multi-scale solution $u_{\varepsilon}$  inherits the medium-scale regularity of the solution $u_{0}$. In this step, we use Lemma \ref{t4.5}, the interior Caccioppoli's inequality for $u_\varepsilon$ and a general  $C^{1,1}$ regularity estimates for $u_{0}$ (see \cite[Lemma 5.3]{MR3974127} for this general $C^{1,1}$ bound).\\

\noindent Step 2. [Iteration.] The previous estimates can be iterated to obtain Lipschitz regularity of $u_{\varepsilon}$ down to scale $\varepsilon$. In this step, we need to notice the scaling property of $\mathcal{L}_{\varepsilon}$ and the interior Caccioppoli's inequality for $u_\varepsilon$.\\

\noindent Step 3. [A blow-up argument] We use the regularity result of $\mathcal{L}_{1}$ to obtain the Lipschitz regularity on scales smaller than $\varepsilon$, since $A\in C^{0,\alpha}$.\\

Since all of the operations above are totally similar to the proofs of \cite[Theorem 4.1.1]{shen2018periodic} and \cite[Theorem 4.1]{MR3974127}, we omit it for simplicity, and we refer to \cite[Theorem 4.1.1]{shen2018periodic} and \cite[Theorem 4.1]{MR3974127} for the details.
\end{proof}

As a direct consequence of the above interior Lipschitz estimates, we deduce the following estimates on the gradient and the mixed gradient of the Green function:

\begin{cor}\label{t4.11}
Let $d\geq 3$. Suppose that the matrix $A$ satisfies the conditions in \eqref{4.2}. Let $\mathcal{G}_\varepsilon$ be the Green function associated with the operator $\operatorname{div}(A(x/\varepsilon)\nabla )$ on $\mathbb{R}^d$. Then, there exists a constant $C$ depending only on $d$ and $A$, such that for any $x\neq y\in \mathbb{R}^d$, there holds\end{cor}
\begin{equation}\label{4.43}
\begin{gathered}
\left|\nabla_{x} \mathcal{G}_\varepsilon(x, y)\right|+\left|\nabla_{y} \mathcal{G}_\varepsilon(x, y)\right| \leq C|x-y|^{-d+1}, \\
\left|\nabla_{x} \nabla_{y} \mathcal{G}_\varepsilon(x, y)\right| \leq C|x-y|^{-d}.
\end{gathered}
\end{equation}

\begin{proof}
The Green function $\mathcal{G}_\varepsilon(x, y)$ associated with the operator $\operatorname{div}(A(x/\varepsilon)\nabla )$ on $\mathbb{R}^d$ is a solution of the following weak formulation equation (see \cite{MR3040889} for a precise definition)
$$\operatorname{div}(A(x/\varepsilon)\nabla_x \mathcal{G}_\varepsilon(x, y))=\delta_y(x)\quad \text{in}\quad \mathbb{R}^d.$$

If $d\geq 3$, since $A$ is uniformly bounded and coercive, it follows from \cite[Theorem 1]{MR3040889} that there exists a unique Green function, satisfying the following estimate:
\begin{equation}\label{4.44}|\mathcal{G}_\varepsilon(x, y)|\leq C|x-y|^{2-d}.\end{equation}

To proceed, we need only to estimate $\nabla_{x} \mathcal{G}_\varepsilon(x, y)$, since the estimate of $\nabla_{y} \mathcal{G}_\varepsilon(x, y)$ follows from
$\mathcal{G}_\varepsilon(x, y)=\mathcal{G}^*_\varepsilon(y, x)$ \cite[Theorem 1.3]{MR657523}, where $\mathcal{G}^*$ is the Green function
associated with the transposed operator $\operatorname{div}(A^*(x/\varepsilon)\nabla )$ on $\mathbb{R}^d$.

Consequently, the estimates \eqref{4.43} follow from Theorem \ref{t4.10}, \eqref{4.44} and the fact that $\mathcal{G}_\varepsilon(, y)$ and $\nabla_y\mathcal{G}_\varepsilon(, y)$ are $A$-harmonic in $B(x,|x-y|/2)$.
\end{proof}

After obtaining the size estimates of $\mathcal{G}_\varepsilon$ and its gradients, we can complete the proof of Theorem \ref{t1.3}, which is stated in the following theorem.
\begin{thm}[Convergence rates II]\label{t4.12}
Under the condition \eqref{4.2}, suppose that $f\in W^{1,p}(\mathbb{R}^d)\cap L^{2d/(d+4)}(\mathbb{R}^d)$ for some $p\in (d,\infty)$, then for any $2<q<\infty$, we have
\begin{equation}\label{4.45}
||w_\varepsilon||_{L^{2d/d-2}(\mathbb{R}^d)}+||\nabla w_\varepsilon||_{L^{q}(\mathbb{R}^d)}\leq C\varepsilon ||f||_{W^{1,p}(\mathbb{R}^d)\cap L^{2d/(d+4)}(\mathbb{R}^d)},
\end{equation}
with $w_\varepsilon$ defined by \eqref{4.38},
where the constant $C$ depends only on $d$, $A$, $m$ and $p$.
\end{thm}

\begin{proof}In view of \eqref{4.41}, we decompose $w_\varepsilon$ as  $w_\varepsilon=:w_{\varepsilon,1}+w_{\varepsilon,2}$, where
\begin{equation}\label{4.46}\begin{aligned}
\operatorname{div}(A^\varepsilon \nabla w_{\varepsilon,1})=&-\varepsilon \nabla_y \varphi_m^\varepsilon\cdot \nabla f;\\
\operatorname{div}(A^\varepsilon \nabla w_{\varepsilon,2})=&\varepsilon \operatorname{div}_x(\nabla_y \varphi_m^\varepsilon\cdot f)+\operatorname{div}\left((\widehat{A}-A^\varepsilon)\cdot \nabla P\cdot (U_0-S_\varepsilon(U_{0}))\right)\\
&+\varepsilon \partial_k(\phi_{kij}^\varepsilon S_\varepsilon(\partial_i U_{0,j}))-\varepsilon \operatorname{div}(A^\varepsilon\chi_j^\varepsilon S_\varepsilon(\nabla U_{0,j})).
\end{aligned}\end{equation}

To proceed, similar to the proof of \eqref{4.29}, it follows the estimates of the  Green functions $\mathcal{G}_\varepsilon$ in \eqref{4.43} and the Hardy-Littlewood-Sobolev inequality that
\begin{equation}\label{4.47}||\nabla w_{\varepsilon,1}||_{L^2(\mathbb{R}^d)}+||\nabla w_{\varepsilon,1}||_{L^\infty(\mathbb{R}^d)}\leq C\varepsilon||\nabla f||_{L^p(\mathbb{R}^d)\cap L^{2d/(d+2)}(\mathbb{R}^d)}.\end{equation}

Moreover, for $2\leq q<\infty$, if follows from the uniform interior Lipschitz estimates Theorem \ref{t4.10} and the so-called real-variable method \cite[Chapter 4.2]{shen2018periodic} that the uniform $W^{1,q}$-estimates holds:
\begin{equation}\label{4.48}\begin{aligned}
||\nabla w_{\varepsilon,2}||_{L^q(\mathbb{R}^d)}&\leq C\varepsilon||f||_{L^q(\mathbb{R}^d)}+C\varepsilon
||\nabla U_0||_{L^q(\mathbb{R}^d)}\\
&\leq C\varepsilon||f||_{L^q(\mathbb{R}^d)\cap L^{2d/(d+4)}(\mathbb{R}^d)}\\
& \leq C\varepsilon ||f||_{W^{1,p}(\mathbb{R}^d)\cap L^{2d/(d+4)}(\mathbb{R}^d)},
\end{aligned}\end{equation}
where we have used  \cite[Proposition 3.1.6]{shen2018periodic}  in the above inequality. One can refer to \cite[Chapter 4.3]{shen2018periodic} for a detailed proof of \eqref{4.48}.

 Therefore, for any $2\leq q<\infty$, we have
\begin{equation}\label{4.49}
||w_\varepsilon||_{L^{2d/d-2}(\mathbb{R}^d)}+||\nabla w_\varepsilon||_{L^{q}(\mathbb{R}^d)}\leq C\varepsilon ||f||_{W^{1,p}(\mathbb{R}^d)\cap L^{2d/(d+4)}(\mathbb{R}^d)}.
\end{equation}
Consequently, we complete this proof, for the constant $C$ depending only on $d$, $A$, $m$ and $p$.

\end{proof}

To complete the proof of Theorem \ref{t1.3}, for some $p>d$, we note that
$$\begin{aligned}||\varepsilon\chi_j^\varepsilon S_\varepsilon(U_{0,j})||_{L^{2d/d-2}(\mathbb{R}^d)\cap L^\infty (\mathbb{R}^d)}&\leq C \varepsilon ||S_\varepsilon(U_{0})||_{L^{2d/d-2}(\mathbb{R}^d)\cap L^\infty (\mathbb{R}^d)}\\
&\leq C \varepsilon ||S_\varepsilon(U_{0})||_{L^{2d/d-2}(\mathbb{R}^d)}+C \varepsilon ||S_\varepsilon(\nabla U_{0})||_{L^{p}(\mathbb{R}^d)}\\
&\leq C \varepsilon ||U_{0}||_{L^{2}(\mathbb{R}^d)}+C \varepsilon ||\nabla U_{0}||_{L^{p}(\mathbb{R}^d)}\\
& \leq C \varepsilon ||f||_{L^{p}(\mathbb{R}^d)\cap L^{2d/(d+4)}(\mathbb{R}^d)},\end{aligned}$$
where we have used $\chi\in L^\infty$, Lemma \ref{t4.7}, Sobolev embedding as well as  \cite[Proposition 3.1.6]{shen2018periodic}  in the above inequality.

Thus, we complete the proof of Theorem \ref{t1.3}.

\begin{center}{\textbf{Acknowledgements}}
\end{center}

The author wants to express his sincere appreciation to  Prof. Wenjia Jing  for suggesting this topic to me and  helpful discussions.
\normalem\bibliographystyle{plain}{}

\end{document}